\documentclass[12pt]{article}
\usepackage{authblk}
\usepackage{graphicx}
\usepackage{subfig}
\usepackage{multirow}
\usepackage{amsfonts}
\usepackage[a4paper,left=2.4cm,right=2.3cm,top=2cm,bottom=1.8cm]{geometry}%
\usepackage{amsmath,amssymb,mathrsfs}
\usepackage{amsthm}
\usepackage{textcomp}
\usepackage{color}

\usepackage{hyperref}
\usepackage{float}
\usepackage{epsfig}
\usepackage{epstopdf}
\usepackage{array,booktabs}
\usepackage{tabularx}
\usepackage{cleveref}
\usepackage{algpseudocode}

\usepackage{longtable}

\usepackage{comment}
\usepackage{bm}
\usepackage[ruled,vlined]{algorithm2e}

\newtheorem{lemma}{Lemma}
\newtheorem{theorem}{Theorem}
\newtheorem{remark}{Remark}

\newtheorem{assumption}{Assumption}
\newtheorem{corollary}{Corollary}

\begin{document}

\title{Global Optimization with A Power-Transformed Objective and Gaussian Smoothing}

\author[1]{Chen Xu}
\affil[1]{Department of Engineering, Shenzhen MSU-BIT University, China. 
\authorcr Email: xuchen@smbu.edu.cn
}

\date{}

\maketitle

\begin{abstract}
We propose a novel method that solves global optimization problems in two steps: (1) perform a (exponential) power-$N$ transformation to the not-necessarily differentiable objective function $f$ and get $f_N$, and (2) optimize the Gaussian-smoothed $f_N$ with stochastic approximations. Under mild conditions on $f$, for any $\delta>0$, we prove that with a sufficiently large power $N_\delta$, this method converges to a solution in the $\delta$-neighborhood of $f$'s global optimum point. The convergence rate is $O(d^2\sigma^4\varepsilon^{-2})$, which is faster than both the standard and single-loop homotopy methods if $\sigma$ is pre-selected to be in $(0,1)$. In most of the experiments performed, our method produces better solutions than other algorithms that also apply smoothing techniques.
\end{abstract}


\section{Introduction}
In this work, we consider the global optimization problem of
\begin{equation}
\label{objective}
\max_{\bm{x}\in\mathcal{S}\subset \mathbb{R}^d} f(\bm{x}),
\end{equation}
where $f(\bm{x})$ is a continuous and possibly non-concave function with a global maximum $f(\bm{x}^*)>\sup_{\bm{x}\neq \bm{x}^*}f(\bm{x})$, and $d$ is a positive integer. The minimize-version of this problem is often encountered in machine learning, such as model training and adversarial attack in computer vision. The gradient-based algorithms, such as the (stochastic) gradient descent, are commonly used, which only guarantee to approximate a locally optimal solution in a general case.

Homotopy, also called graduated continuation, is a class of methods for finding a global solution to (\ref{objective}), with many successful applications in machine learning (e.g.,\cite{Xu-NIPS2016, Iwakiri2022}). It converts the original problem to
\begin{equation}
\label{homotopy-obj}
\max_{\bm{\mu}\in\mathbb{R}^d,\sigma\geq0} \mathbb{E}_{\bm{\xi}}[f(\bm{\mu}+\sigma\bm{\xi})],
\end{equation}
where $\sigma\geq 0$ is called the scaling coefficient and $\xi$ is a random variable with a pre-selected distribution, such as a standard multivariate Gaussian distribution (Gaussian Homotopy, GH) or a uniform distribution in a unit sphere. Based on the observation that $\bm{\mu}^*_{\sigma}:=\arg\max_{\bm{\mu}} \mathbb{E}[f(\bm{\mu}+\sigma\xi)]$ approaches\footnote{Note that $\mathbb{E}[f(\bm{\mu}+\sigma\xi)]=f(\bm{\mu})$ if $\sigma=0$.} $\bm{x}^*$ as $\sigma$ decreases to 0, the homotopy methods admits a double-loop mechanism: the outer loop iteratively decreases $\sigma$, and for each fixed value of $\sigma$, the inner loop solves $\max_{\bm{\mu}} \mathbb{E}[f(\bm{\mu}+\sigma\xi)]$, with the solution found in the current inner loop as the starting search point in the next inner loop.

The double-loop mechanism of the homotopy method is costly in time. To tackle this issue, \cite{Iwakiri2022} propose a single-loop Gaussian homotopy (SLGH) method that iteratively performs one-step update of $\bm{\mu}$ and $\sigma$, which reduces the convergence rate from $O(d^2/\epsilon^4)$ to $O(d/\epsilon^2)$. However, in theory, SLGH only guarantees to approximate a local optimum\footnote{Theorem 4.1 in \cite{Iwakiri2022} shows that SLGH approximates a solution $\hat{\bm{x}}$ such that $\mathbb{E}[\nabla f(\hat{\bm{x}})]=0$.}, which is not necessarily a global one. A time-efficient algorithm that aims at the global maximum is still to be found. 

Therefore, in this work, we propose a new method, namely the Gaussian Smoothing with a Power-transformed Objective (GSPTO),  for solving the optimization problem of (\ref{objective}). According to our Corollary \ref{summary}, GSPTO converges to a neighborhood of $\bm{x}^*$ with the rate of $O(d^2\sigma^4/\epsilon^{2})$. It indicates that GSPTO is faster than the standard homotopy and SLGH, if $\sigma>0$ is pre-selected to be in $(0,1)$. Most of the experiments in Section \ref{experiments} show that the GSPTO-based algorithm (e.g., EPGS, introduced later) outperforms other algorithms that also apply the smoothing technique.


\subsubsection*{Motivation}
Under the condition of $\int_{\mathbb{R}^d} f(\bm{x}) d\bm{x}>0$ and an additional one, there is a threshold $\sigma_m>0$ such that whenever $\sigma>\sigma_m$, the Gaussian-smoothed objective $E_{\bm{\xi}\sim\mathcal{N}(\bm{0},I_d)}[f(\bm{\mu}+\sigma \bm{\xi})]$ is concave in $\bm{\mu}$ (see \cite[Main Result (Corollary 9)]{Mobahi2012}). Hence, Gaussian smoothing converts the original possibly non-concave maximization problem to a concave one, if the maximum point $\bm{\mu}^*$ coincides with $\bm{x}^*$. Although this condition is not true in general\footnote{This is why smoothing alone is insufficient for optimization, and is typically used in conjunction with iteratively reducing the scaling parameter, which becomes the homotopy algorithm.}, we can modify the objective to make $\bm{\mu}^*$ close to $\bm{x}^*$, where recall that $\bm{x}^*$ denotes the global maximum point of the original objective $f$ before modification.

Intuitively, if we modify $f(\bm{x})$ to put sufficiently large weight on its global maximum $\bm{x}^*$, the global maximum $\bm{\mu}^*:=\arg\max_{\bm{\mu}} F(\bm{\mu},\sigma)$ should get close enough to $\bm{x}^*$. One way of such modification is by taking powers of $f$, if $f(\bm{x}^*)>1$. The difference $f^N(\bm{x}^*)-f^N(\bm{x})$ is positively related with the power $N$, which indicates that more weight is put on $\bm{x}^*$ as $N$ increases. Figure 1(a) verifies this intuition with an example, and Figure 1(b) illustrates the effects of taking exponential powers. As shown in these two toy examples, $\bm{\mu}^*$ approaches $\bm{x}^*$ as $N$ increases.

\begin{figure}[h]
\begin{tabular}{cc}
	\centering
        \subfloat[Gaussian Smoothing of $f^N$.]{\includegraphics[scale=0.55]{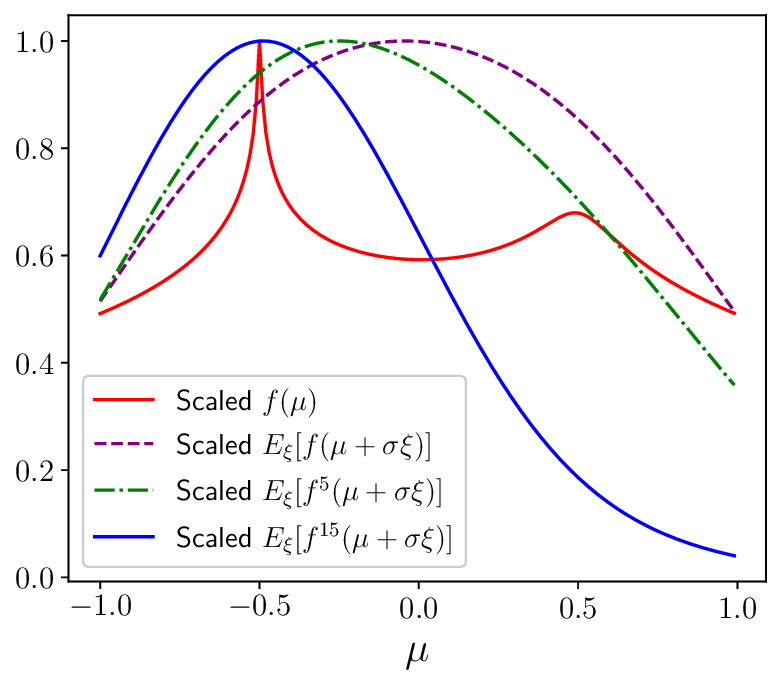} }
        &
        \subfloat[Gaussian Smoothing of $\exp(Nf)$.]{\includegraphics[scale=0.55]{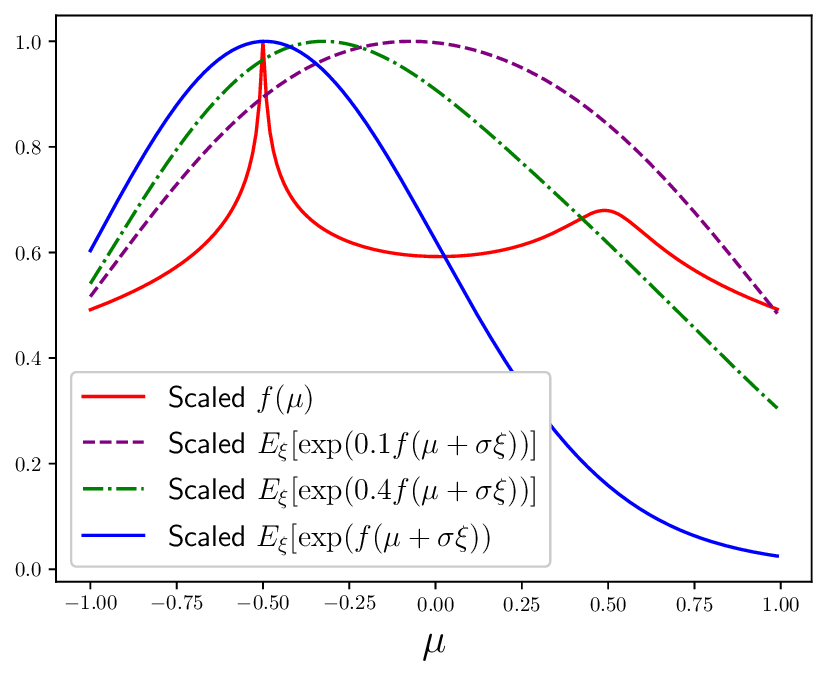} }
\end{tabular}
       	\caption{Effects of elevating the objective $f$ before Gaussian smoothing (A toy example): The maximum point of $F_N(\mu):=\mathbb{E}_{\xi\sim\mathcal{N}(0,1)}[f_N(\mu+\sigma\xi)]$ gets closer to the global maximum point $f$ as $N$ increases, where $\sigma=0.5$ and $f(\mu)=-\log((\mu+0.5)^2+10^{-5})-\log((\mu+0.5)^2+10^{-2})+10$ for $|\mu|\leq 1$ and $f(\mu)=0$ for $|\mu|>1.$ For easier comparison, the graph of each function is scaled to have a maximum value of 1.} 
\label{escape-local2D}
\end{figure}

From the above intuition, we propose GSPTO for solving the global optimization problem (\ref{objective}), which is a new method that places more weight on the objective $f$'s maximum value (by increasing the gap between the global and local maximum values) before performing Gaussian smoothing. Based on GSPTO, we design two algorithms\footnote{See Algorithm \ref{alg1} for the two algorithms.}, Power Gaussian Smoothing (PGS) and Exponential Power Gaussian Smoothing (EPGS), which are featured with replacing the original objective $f(\bm{x})$ with a (exponential) power transformation. Specifically, with $\sigma$ and $N$ as two hyper-parameters, PGS solves $\max_{\bm{\mu}} \mathbb{E}_{\bm{x}\sim \mathcal{N}(\bm{\mu},\sigma^2 I_d)}[f^N(X)]$ and EPGS solves $\max_{\bm{\mu}} \mathbb{E}_{\bm{x}\sim \mathcal{N}(\bm{\mu},\sigma^2 I_d)}[e^{Nf(\bm{x})}]$, both using a stochastic gradient ascent algorithm derived in this paper, which does not require the differentiability of $f$. Here, $\mathcal{N}$ denotes a multivariate Gaussian distribution and $I_d$ denotes an identity matrix of dimension $d\times d$.

\subsubsection*{Related Work}
The homotopy methods, firstly proposed in \cite[Chapter 7]{BlakeZisserman1987}, are intensively studied in the field of machine learning for global optimization problems. \cite{MobahiFisher2015} derives a bound for the worst scenario of the GH algorithm in a deterministic setting (i.e., the expectation $\mathbb{E}$ is not approximated with samples), while \cite{Hazan2016} provides a convergence analysis in a stochastic setting (i.e., $\mathbb{E}$ is estimated with samples). Specifically, the latter proves that with a probability greater than $1-p$, the solution $\hat{\bm{x}}$ produced by their proposed homotopy algorithm is $\epsilon$-optimal (i.e., $f(\bm{x}^*)-f(\hat{\bm{x}})<\epsilon$) after $\tilde{O}(d^2/(\sigma^2\epsilon^4)$ steps of solution-update. \cite{GaoSener2022} changes the distribution of the perturbation $\xi$ from the commonly used Gaussian or uniform to the distribution that minimizes the estimation error of the gradient $\nabla_{\bm{\mu}}E_{\bm{\xi}\sim\mathcal{N}(\bm{0},I_d)}[f(\bm{\mu}+\sigma \bm{\xi})]$. \cite{Lin2023} proposes an algorithm for learning the whole solution path produced by the homotopy. Specifically, their algorithm learns a model $x_\phi (\sigma)$ that predicts (for any $\sigma>0$) the solution to $\min_{\bm{\mu}}\nabla_{\bm{\mu}}E_{\bm{\xi}\sim\mathcal{N}(\bm{0},I_d)}[f(\bm{\mu}+\sigma \bm{\xi})]$, where $\phi$ is the set of model parameters to be trained. 

The smoothing and homotopy methods have a large number of successful applications in machine learning, such as neural network training (\cite{Hazan2016}), adversarial attack on image classification models (\cite{Iwakiri2022}), solving $L_1$-regularized least-square problems (\cite{Xiao2012}), neural combinatorial optimization (\cite{GaoSener2022}), improving the optimization algorithms of stochastic gradient descent and Adam (\cite{StarnesWebster2024}), and so on.

There are a few existing studies (\cite{Dvijotham2014,Roulet2020,Chen2024}) that replace the original $f$ with a surrogate objective that also involves the exponential transformation $e^{Nf(\bm{\mu}+\xi)}$ before smoothing. But their works are different from ours. \cite{Dvijotham2014} propose to minimize the surrogate objective of  $G(\bm{\mu}):=\frac{1}{N}\log\left(\mathbb{E}_{\xi\sim\mathcal{N}(0,\Sigma)}[e^{Nf(\bm{\mu}+\xi)}] \right)+\frac{1}{2}\bm{\mu}^T R\bm{\mu}$. The theory\footnote{$\min_{\mu} G(\bm{\mu},\Sigma)$ is a convex problem given that $N R-\Sigma^{-1}$ is positive definite and $\mathcal{C}$ is convex (\cite[Theorem 3.1]{Dvijotham2014}).} that justifies this surrogate objective requires that  $N,R$, and $\Sigma$ be selected so that $N R-\Sigma^{-1}$ is positive semi-definite. This indicates that EPGS, for which $R=\bm{0}$ and $\Sigma = \sigma I_d$, is not a special case of theirs, since $-\sigma^{-1} I_d$ is negative definite and violates their requirement. Moreover, their theory on the distance between the optimum point of the new surrogate and $\bm{x}^*$ is incomplete (see \cite[Section 3.2]{Dvijotham2014}). For optimal-control problems, \cite{Roulet2020} study the surrogate objective that is similar to $G(\bm{\mu})$, and provide a theoretical analysis on the corresponding algorithm's convergence to a stationary point. However, the relation between this stationary point and the global optimum point $\bm{x}^*$ is not revealed. The proposed surrogate objective in \cite{Chen2024} is $(1-N)f(\bm{\mu})+\log\left(\mathbb{E}_{\xi\sim\mathcal{N}(0,\Sigma)}[e^{Nf(\bm{\mu}+\xi)}] \right)$, where $N\in[0,1]$. This is very different from GSPTO's requirement that $N$ is sufficiently large. Also, their theory (\cite[Theorem 10]{Chen2024}) bounds $|f(\bm{x}^*)-f(\bm{\mu}^*)|$ with $O(N\sigma^2)$, which does not imply an improvement for increasing the value of $N$.

 
\subsection*{Contribution}
This paper introduces a novel method, GSPTO, for solving global optimization problems, with the contributions summarized as follows.
\begin{enumerate}
\item To our knowledge, for global optimization problems, this is the first work that proposes the idea\footnote{\cite{Dvijotham2014}, \cite{Roulet2020}, and \cite{Chen2024} have not mentioned this idea.} of putting sufficiently large weight on the global maximum values of the objective, to decrease the distance between the optimum point before and after Gaussian smoothing (i.e., $\|\bm{x}^*-\bm{\mu}^*\|$). PGS and EPGS are two ways of realizing this idea, and future studies are motivated for finding better ones.

\item In theory, GSPTO is faster than the homotopy methods for optimization. According to Corollary \ref{summary}, GSPTO has a convergence rate of $O(d^2\sigma^4\varepsilon^{-2})$, which is faster than the standard homotopy method ($O(d^2\sigma^{-2}\varepsilon^{-4})$,\cite[Theorem 5.1]{Hazan2016}), and SLGH ($O(d^2\varepsilon^{-2})$,\cite[Theorem 4.1]{Iwakiri2022}), if $\sigma$ is pre-selected to be in $(0,1)$. Most of our experiments show that it outperforms other optimization algorithms that also apply the smoothing technique (see Section \ref{ExperimentSummary}). 

\item Our convergence analysis does not require the Lipschitz condition on the original objective $f$, which is assumed in the theoretical analysis of homotopy methods in other studies (\cite{Hazan2016,Iwakiri2022}). Therefore, our analysis applies to more situations.
\item The theory derived in this work is on the distance between the found solution and the optimal one, while the convergence analysis in other studies on homotopy (e.g., \cite{Hazan2016,Iwakiri2022}) is on the objective value of the found solution. Therefore, our theory has a wider range of applications (e.g., for problems that concern the distance between the found solution and the optimal like inverse problems and adversarial attack in image recognition).
\end{enumerate}

\section{Preliminaries}
We rigorously prove the intuition that motivates GSPTO: Given $\sigma>0$, for any $\delta>0$, there exists a threshold such that whenever $N$ exceeds this threshold, the global maximum point of $F_N(\bm{\mu},\sigma)$ lies within a $\delta$-neighborhood of $\bm{x}^*$, where $F_N(\bm{\mu}):=\mathbb{E}_{\bm{x}\sim\mathcal{N}(\bm{\mu},\sigma I_d)}[f_N(\bm{x})]$, and
\begin{equation}
\begin{split}
\label{fN}
&f_N(\bm{x}_k):=
\left\{
\begin{array}{ll}
f^N(\bm{x}_k), &\bm{x}\in\mathcal{S};\\
0, &\text{otherwise},\\
\end{array}
\right. \text{(PGS setting);}\; \\
&f_N(\bm{x}_k):=
\left\{
\begin{array}{ll}
e^{Nf(\bm{x}_k)}, &\bm{x}\in\mathcal{S};\\
0, &\text{otherwise}.\\
\end{array}
\right.
\text{(EPGS setting).}
\end{split}
\end{equation}

\begin{theorem}
\label{Thm1}
Let $f:\mathcal{S}\subset\mathbb{R}^d\rightarrow \mathbb{R}$ be a continuous function that is possibly non-concave (and non-negative only for the case of PGS), where $\mathcal{S}$ is compact. Assume that $f$ has a global maximum $\bm{x}^*$ such that $\sup_{\bm{x}: \lVert \bm{x} - \bm{x}^*\rVert \geq \delta}f(\bm{x})<f(\bm{x}^*)$ for any $\delta>0$. For $\sigma>0$, define 
\begin{equation*}
F_N(\bm{\mu},\sigma) := (\sqrt{2\pi}\sigma)^{-k} \int_{\bm{x}\in \mathbb{R}^k} f_N(\bm{x}) e^{-\frac{\lVert \bm{x} - \bm{\mu} \rVert^2}{2\sigma^2}} d\bm{x}, \quad N\in\{0,1,2,... \},
\end{equation*}
where $f_N$ is defined in (\ref{fN}) for either PGS or EPGS. Then, for any $M>0$ and $\delta>0$ such that $\bar{B}(\bm{x}^*;\delta):=\{\bm{x}\in\mathbb{R}^d: \|\bm{x}-\bm{x}^*\|\leq\delta\}\subset\mathcal{S}$, there exists $N_{\delta,\sigma,M}>0$, such that whenever $N>N_{\delta,\sigma,M}$, we have for any $\|\bm{\mu}\|<M$ that: $\frac{\partial F_{N}(\bm{\mu},\sigma)}{\partial \mu_i}>0$ if $\mu_i<x_i^*-\delta$, and $\frac{\partial F_{N}(\bm{\mu},\sigma)}{\partial \mu_i}<0$ if $\mu_i>x_i^*+\delta$. Here, $\mu_i$ and $x_i^*$ denote the $i^{th}$ entry of $\bm{\mu}$ and $\bm{x}^*$, respectively, where $i\in\{1,2,...,d\}$.
\end{theorem}
\begin{proof}
See the appendix for the proof for the EPGS setting. The proof for the PGS setting is similar.
\end{proof}

\section{Gaussian Smoothing with Power-Transformed Objective}
\subsection{The Solution Updating Rule}
For the optimization problem (\ref{objective}), based on Theorem \ref{Thm1}, with the pre-selected hyper-parameters $N$ and $\sigma>0$, GSPTO follows a stochastic gradient ascent scheme to solve $\max_{\bm{\mu}}F_N(\bm{\mu},\sigma)$. Specifically, the rule for updating the solution candidate used is
\begin{equation}
\label{sga}
\text{GSPTO}:\qquad \bm{\mu}_{t+1} = \bm{\mu}_{t} + \alpha_t \hat{\nabla} F_N(\bm{\mu}_t),
\end{equation}
where $\hat{\nabla}_{\bm{\mu}} F(\bm{\mu}_t):=\frac{1}{K} \sum_{k=1}^K (\bm{x}_k-\bm{\mu}_t)f_N(\bm{x}_k)$, $\{\bm{x}_k\}_{k=1}^K$ are independently sampled from the multivariate Gaussian distribution $\mathcal{N}(\bm{\mu},\sigma^2 I_d)$, and $f_N(\bm{x}_k)$ is defined in (\ref{fN}). Note that $\widehat{\nabla F_N}(\bm{\mu}_t)$ is a sample estimate of the gradient $\nabla F(\bm{\mu}_t)$:
\begin{equation*}
\begin{split}
\label{nablaF}
\nabla_{\bm{\mu}} F(\bm{\mu}_t) &= \nabla_{\bm{\mu}}  \mathbb{E}_{\bm{x}\sim \mathcal{N}(\bm{\mu},\sigma^2 I_d)}[f_N(\bm{x})] \\
&=  (\sqrt{2\pi}\sigma)^{-1} \int_{\bm{x}\in B(\bm{0};M)} (\bm{x}-\bm{\mu}_t)f_N(\bm{x}) e^{-\frac{\lVert \bm{x} - \bm{\mu}_t \rVert^2}{2\sigma^2}} d\bm{x}.\\
&= \mathbb{E}_{\bm{x\sim \mathcal{N}(\bm{\mu},\sigma I_d)}}[(\bm{x}-\bm{\mu}_t)f_N(\bm{x})].
\end{split}
\end{equation*}
Based on GSPTO, PGS and EPGS are designed in Algorithm \ref{alg1}. They normalize the gradient before updating the solution, which is a common practice to stabilize results.
\begin{algorithm}[h]
  \caption{PGS/EPGS}\label{alg1}
  \SetAlgoLined
  \textbf{Require}: {The power $N>0$, the scaling parameter $\sigma>0$, the objective $f$, the initial value $\bm{\mu}_0$, the number $K$ of sampled points for gradient approximation, the total number $T$ of $\bm{\mu}$-updates, and the learning rate schedule $\{\alpha_t\}_{t=1}^T$.}

  \textbf{Result}: {$\bm{\mu}_T$ - The approximated solution to (\ref{objective}).}

  \For{t from 0 to T-1} {
    Independently sample from $\mathcal{N}(\bm{\mu}_{t},\sigma I_d)$ (multivariate Gaussian distribution), and obtain $\{\bm{x}_k\}_{k=1}^K$.\\ 
    $ \bm{\mu}_{t+1} =\bm{\mu}_{t} + \alpha_t \widehat{\nabla F_N}(\bm{\mu}_t)/\|\widehat{\nabla F_N}(\bm{\mu}_t)\|$, where $\hat{\nabla}_{\bm{\mu}} F(\bm{\mu}_t):=\frac{1}{K} \sum_{k=1}^K (\bm{x}_k-\bm{\mu}_t)f_N(\bm{x}_k)$ and
$f_N(\bm{x}_k):=
\left\{
\begin{array}{ll}
f^N(\bm{x}_k), &\text{for PGS};\\
e^{Nf(\bm{x}_k)}, &\text{for EPGS}.\\
\end{array}
\right.
$
}
  Return($\bm{\mu}_T$).
\end{algorithm}
\begin{remark}
An effective method to avoid computation overflows caused by a large $N$-value is to modify the gradient estimate as
$$\hat{\nabla}_{\bm{\mu}} F(\bm{\mu}_t):=
\left\{
\begin{array}{ll}
\frac{1}{K} \sum_{k=1}^K (\bm{x}_k-\bm{\mu}_k)f_N(\bm{x}_k)f^N(\bm{x}_k)/f^N(\bm{\mu}_t), &\text{for PGS};\\
\frac{1}{K} \sum_{k=1}^K (\bm{x}_k-\bm{\mu}_k)f_N(\bm{x}_k)e^{N(f(\bm{x}_k)-f(\bm{\mu}_t))}, &\text{for EPGS}.\\
\end{array}
\right.
$$
This modification produces the same $\bm{\mu}$-updates as in Algorithm \ref{alg1} because of the gradient-normalization step.
\end{remark}
\section{Convergence Analysis}
We perform convergence analysis for the updating rule (\ref{sga}) under the PGS and EPGS setting (\ref{fN}) on the optimization problem of ($\ref{objective}$). We show that, for any $\varepsilon>0$ and any $\delta>0$, GSPTO converges to a $\delta$-neighborhood of $\bm{x}^*$ with the iteration complexity of $O(d^2\sigma^4\varepsilon^{-2})$. Specifically, with $T=O((d\sigma^2)^{2/(1-2\gamma)}\varepsilon^{-2/(1-2\gamma)})$ times of updating $\bm{\mu}_t$, $\mathbb{E}[\|\nabla F(\bm{\mu}_T)\|^2]<\varepsilon$, where $\gamma$ can be arbitrarily close to 0. The result is summarized in Corollary \ref{summary}.

\subsection{Notation}
\label{notation}
In this section, let $\delta>0$ and $\sigma>0$ be fixed, and $\delta>0$ satisfies
\begin{equation}
\label{delta-neighborhood}
\mathcal{S}_{\bm{x}^*,\delta}:=\{\bm{\mu}\in\mathbb{R}^d: |\mu_i-x_i^*|\leq \delta, \text{ for each } i\in\{1,2,...,d\} \}\subset \mathcal{S},
\end{equation}
where $i$ denotes the $i^{th}$ entry and $\mathcal{S}$ is specified in Assumption \ref{f-bound}. Let $N>0$ be such that for any $\|\bm{\mu}\|\leq\sqrt{d}M$: $\frac{\partial F_{N}(\bm{\mu},\sigma)}{\partial \mu_i}>0$ if $\mu_i<x_i^*-\delta$, and $\frac{\partial F_{N}(\bm{\mu},\sigma)}{\partial \mu_i}<0$ if $\mu_i>x_i^*+\delta$, for all $i\in\{1,2,...,d\}$. Such an $N$ exists because of Theorem \ref{Thm1}. Here, $M$ is specified in Assumption \ref{f-bound}. Let $F_N(\bm{\mu},\sigma)$ be as defined as in Theorem \ref{Thm1}. Unless needed, we omit $N$ and $\sigma$ in this symbol as they remain fixed in this section, and write $F(\bm{\mu})$ instead. $\nabla F(\bm{\mu})$ refers to the gradient of $F$ with respect to $\bm{\mu}$. Let $f_N$ be defined as in (\ref{fN}).

\subsection{Assumptions and Lemmas}
\begin{assumption}
\label{f-bound}
Assume that $f(\bm{x}):\mathcal{S}\rightarrow \mathbb{R}$ is a function satisfying the conditions specified in Theorem \ref{Thm1}, where $\mathcal{S}\subset\mathcal{S}_M:=\{\bm{x}\in \mathbb{R}^d: |x_i|\leq M, i\in\{1,2,...,d\}\}$, for some $M>0$. Also, $\|F(\bm{\mu})\|,\| \nabla F(\bm{\mu})\|<+\infty$ for all $\bm{\mu}\in \mathbb{R}^d$.

\end{assumption}

\begin{assumption}
\label{learning-rates}
Assume that the learning rate $\alpha_t$ satisfies 
$$\alpha_t>0,\; \sum_{t=0}^{\infty}\alpha_t = +\infty, \text{ and } \sum_{t=0}^\infty \alpha_t^2<+\infty.$$
\end{assumption}

\begin{lemma}
Under Assumption \ref{f-bound}, any local or global maximum point $\bm{\mu}^*$ of $F(\bm{\mu})$ belongs to $\mathcal{S}_{\bm{x}^*,\delta}$, which is defined in (\ref{delta-neighborhood}). 
\end{lemma}
\begin{proof}
For any point $\bm{\mu}\notin \mathcal{S}_{\bm{x}^*,\delta}$, we show that $\nabla F(\bm{\mu})\neq 0$. If $\bm{\mu}\in \mathcal{S}_M - \mathcal{S}_{\bm{x}^*,\delta}$, then $\|\bm{\mu}\|\leq\sqrt{d}M$ and there is some $j\in\{1,2,...,d\}$ such that $|\mu_j-x_j^*|>\delta$, which implies $\frac{\partial F(\bm{\mu})}{\partial \mu_j}\neq 0$ because of the definition of $N$ in Section \ref{notation}.  

On the other hand, if $\bm{\mu}\notin \mathcal{S}_M$, there is at least one $j$ such that $|\mu_j|>M$. Then,
\begin{equation*}
\begin{split}
\frac{\partial F(\bm{\mu})}{\partial \mu_j} &= \frac{1}{(\sqrt{2\pi})^k\sigma^{k+2}}\int_{\bm{x}\in \mathbb{R}^d} (x_j - \mu_j) e^{-\frac{\lVert \bm{x} - \bm{\mu} \rVert^2}{2\sigma^2}} f_N(\bm{x}) d\bm{x}\\
&= \frac{1}{(\sqrt{2\pi})^k\sigma^{k+2}}\int_{\bm{x}\in \mathcal{S}} (x_j - \mu_j) e^{-\frac{\lVert \bm{x} - \bm{\mu} \rVert^2}{2\sigma^2}} f_N(\bm{x}) d\bm{x},\text{ by def. of }f_N\text{ in (\ref{fN})},\\
&=
\left\{
\begin{array}{ll}
negative, &\text{ if } \mu_j>M;\\
positive, &\text{ if } \mu_j<-M.\\
\end{array}
\right., \text{ since }x_j\in\mathcal{S}\Rightarrow |x_j|\leq M \text{ by Assumption \ref{f-bound}}.
\end{split}
\end{equation*}
In sum, for any point $\bm{\mu}\notin \mathcal{S}_{\bm{x}^*,\delta}$, $\nabla F(\bm{\mu})\neq 0$, which further implies that any local or global maximum point $\bm{\mu}^*$ of $F(\bm{\mu})$ belongs to $\mathcal{S}_{\bm{x}^*,\delta}$ since $\nabla F(\bm{\mu}^*)= 0$.
\end{proof}

\begin{lemma}
\label{Lipschitz}
Under Assumption \ref{f-bound}, for any $\sigma>0$, the objective function $F_N(\bm{\mu},\sigma)$ is Lipschitz Smooth. That is, for any $\bm{\mu}_1, \bm{\mu}_2\in \mathbb{R}^d$,
$$ \| \nabla F_N(\bm{\mu}_1,\sigma) - \nabla F_N(\bm{\mu}_2,\sigma) \| \leq L \| \bm{\mu}_1-\bm{\mu}_2 \|, $$
where $L=f^N(\bm{x}^*)$ for the case of PGS and $L=e^{Nf(\bm{x}^*)}$ for the case of EPGS.
\end{lemma}
\begin{proof}
$F_N(\bm{\mu},\sigma)=\mathbb{E}_{\bm{x}\sim\mathcal{N}(\bm{\mu},\sigma^2 I_d)}[f_N(\bm{x})]$, where $f_N(\bm{x})$ is as defined in (\ref{fN}). $\nabla_{\bm{\mu}} F_N(\bm{\mu},\sigma)=\mathbb{E}_{\bm{x}\sim\mathcal{N}(\bm{\mu},\sigma^2 I_d)}[(\bm{x}-\bm{\mu})f_N(\bm{x})]$.
Then,
$$\| \nabla F_N(\bm{\mu}_1,\sigma) - \nabla F_N(\bm{\mu}_2,\sigma) \|= \|\mathbb{E}_{\bm{x}\sim\mathcal{N}(\bm{\mu},\sigma^2 I_d)}[(\bm{\mu}_1-\bm{\mu}_2)f_N(\bm{x})]\|\leq \|\bm{\mu}_1-\bm{\mu}_2 \|f_N(\bm{x}^*).$$
Hence, $L=f_N(\bm{x}^*)$, which is $f^N(\bm{x}^*)$ for PGS and $L=e^{Nf(\bm{x}^*)}$ for EPGS.
\end{proof}

\begin{lemma}
\label{variance-bound}
Under Assumption \ref{f-bound}, for any $\sigma>0$, let $\hat{\nabla}_{\bm{\mu}}F_N(\bm{\mu},\sigma)$ be defined in (\ref{sga}). Then,
$\mathbb{E}[\|\hat{\nabla}_{\bm{\mu}}F_N(\bm{\mu},\sigma)\|^2]<G$, where
\begin{equation*}
\begin{split}
G=\left\{
\begin{array}{ll}
d\sigma^2 f^{2N}(\bm{x}^*),& \text{ for PGS};\\
d\sigma^2 e^{2Nf(\bm{x}^*)},& \text{ for EPGS}.
\end{array}
\right.
\end{split}
\end{equation*}
\end{lemma}
\begin{proof}
\begin{equation*}
\begin{split}
\mathbb{E}\left[ \|\hat{\nabla}F_{N}(\bm{\mu}_t)\|^2\right] &= \frac{1}{K^2} \sum_{k=1}^K \sum_{l=1}^K \mathbb{E}_{\bm{x}_k,\bm{x}_l\sim \mathcal{N}(\bm{\mu}_t,\sigma^2 I_d)}[(\bm{x}_k-\bm{\mu}_{t})'(\bm{x}_l-\bm{\mu}_{t})f_N(\bm{x}_k)f_N(\bm{x}_l)]\\
&\leq f^2_N(\bm{x}^*) \frac{1}{K^2} \sum_{k=1}^K \sum_{l=1}^K \mathbb{E}_{\bm{x}_k,\bm{x}_l\sim \mathcal{N}(\bm{\mu}_t,\sigma^2 I_d)}[\left|(\bm{x}_k-\bm{\mu}_{t})'(\bm{x}_l-\bm{\mu}_{t})\right|]\\
&\leq f^2_N(\bm{x}^*)  \frac{1}{K^2} \sum_{k=1}^K \sum_{l=1}^K \mathbb{E}_{\bm{x}_k,\bm{x}_l\sim \mathcal{N}(\bm{\mu}_t,\sigma^2 I_d)}[\|(\bm{x}_k-\bm{\mu}_{t})\|\cdot\|(\bm{x}_l-\bm{\mu}_{t})\|],\\
&\leq  f^2_N(\bm{x}^*) \frac{1}{K^2} \sum_{k=1}^K \sum_{l=1}^K \sqrt{ \mathbb{E}[\|\bm{x}_{k}-\bm{\mu}_{t}\|^2] \mathbb{E} [\|\bm{x}_{l}-\bm{\mu}_{t})\|^2]},\quad \text{by Schwarz Ineq.,}\\
&=f^2_N(\bm{x}^*) \sigma^2d,\quad d\text{ denotes the number of dimensions},
\end{split}
\end{equation*}
where the third line is by Cauchy-Schwarz Inequality. Replacing $f^2_N(\bm{x}^*)=e^{2Nf(\bm{x}^*)}$ for EPGS and $f^2_N(\bm{x}^*)=f^{2N}(\bm{x}^*)$ for PGS gives the desired result.
\end{proof}

\subsection{Convergence Rate}
\begin{theorem}
\label{convergence-rate}
Let $\{\bm{\mu}_t\}_{t=0}^T\subset \mathbb{R}^d$ be produced by following the iteration rule of (\ref{sga}), with a pre-selected and deterministic $\bm{\mu}_0$ and all the involved terms defined as in Section \ref{notation}. Then, under Assumption \ref{f-bound} and \ref{learning-rates}, we have that
$$\sum_{t=0}^{T-1} \alpha_t \mathbb{E}[\|\nabla F(\bm{\mu}_t)\|^2] \leq f(\bm{x}^*) - F(\bm{\mu}_{0}) + L G \sum_{t=0}^\infty \alpha_t^2,$$
where $L$ and $G$ are as defined in Lemma \ref{Lipschitz} and Lemma \ref{variance-bound}, respectively.
\end{theorem}
\begin{remark}
If $\alpha_t=(t+1)^{-(1/2+\gamma)}$ with $\gamma\in(0,1/2)$. Then, 
$$ \frac{f(\bm{x}^*) - F(\bm{\mu}_{0}) + L G \sum_{t=1}^\infty t^{-(1+2\gamma)}}{\sum_{t=1}^{T} t^{-(1/2+\gamma)}}<\frac{f(\bm{x}^*) - F(\bm{\mu}_{0}) + L G \sum_{t=1}^\infty t^{-(1+2\gamma)}}{\int_{1}^Tt^{-(1/2+\gamma)}dt }=O(d\sigma^2T^{-1/2+\gamma}),$$
where $d\sigma^2$ comes from $G$ in Lemma \ref{variance-bound}. This inequality and Theorem \ref{convergence-rate} implies that after $T=O((d\sigma^2\varepsilon^{{-1}})^{2/(1-2\gamma)})$ times of updating $\bm{\mu}_t$ by GSPTO, $\mathbb{E}[\|\nabla F(\bm{\mu}_T)\|^2]<\varepsilon$. In sum, the GSPTO method (\ref{sga}) converges to a $\delta$-neighborhood of $\bm{x}^*$ with a rate of $O((d\sigma^2\varepsilon^{{-1}})^{2/(1-2\gamma)})$, where $\gamma$ can be arbitrarily close to 0. 
\end{remark}
\begin{proof}
By the Gradient Mean Value Theorem, there exists $\bm{\nu}_t\in\mathbb{R}^d$ such that $\nu_{t,i}$ lies between $\mu_{t+1,i}$ and $\mu_{t,i}$ for each of the $i$th entry, and
\begin{equation*}
\begin{split}
F(\bm{\mu}_{t+1})&=F(\bm{\mu}_t) + (\nabla F(\bm{\nu}_t))' (\bm{\mu}_{t+1}-\bm{\mu}_t),\quad ' \text{ denotes matrix transpose,}\\
&=F(\bm{\mu}_t) + (\nabla F(\bm{\mu}_t))' (\bm{\mu}_{t+1}-\bm{\mu}_t) + (\nabla F(\bm{\nu}_t)-\nabla F(\bm{\mu}_t))'(\bm{\mu}_{t+1}-\bm{\mu}_t)\\
&=F(\bm{\mu}_t) + \alpha_t (\nabla F(\bm{\mu}_t))' (\hat{\nabla} F(\bm{\mu}_t)) - (\nabla F(\bm{\mu}_t)-\nabla F(\bm{\nu}_t))'(\bm{\mu}_{t+1}-\bm{\mu}_t)\\
&\geq F(\bm{\mu}_t) + \alpha_t (\nabla F(\bm{\mu}_t))' (\hat{\nabla} F(\bm{\mu}_t))-L\|\bm{v}_t-\bm{\mu}_t\|\cdot\|\bm{\mu}_{t+1}-\bm{\mu}_t \|, \;\text{by Lemma \ref{Lipschitz}},\\
&\geq F(\bm{\mu}_t) + \alpha_t (\nabla F(\bm{\mu}_t))' (\hat{\nabla} F(\bm{\mu}_t))-L\|\bm{\mu}_{t+1}-\bm{\mu}_t \|^2,\;\text{$v_{t,i}$ is between $\mu_{t+1,i}$ and $\mu_{t,i}$}\\
&= F(\bm{\mu}_t) + \alpha_t (\nabla F(\bm{\mu}_t))' (\hat{\nabla} F(\bm{\mu}_t))-\alpha_t^2 L\|\hat{\nabla} F_N(\bm{\mu}_t)\|^2.
\end{split}
\end{equation*}
Hence, we have
$$ F(\bm{\mu}_{t+1})\geq   F(\bm{\mu}_t) + \alpha_t (\nabla F(\bm{\mu}_t))' (\hat{\nabla} F(\bm{\mu}_t))-\alpha_t^2 L\|\hat{\nabla} F_N(\bm{\mu}_t)\|^2.$$
Taking the expectation of both sides gives
\begin{equation}
\begin{split}
\label{iterate}
\mathbb{E}[F(\bm{\mu}_{t+1})] &\geq \mathbb{E}[F(\bm{\mu}_{t})]+\alpha_t\mathbb{E}[ \|\nabla F(\bm{\mu}_t)) \|^2] - \alpha_t^2 L \mathbb{E}[\| \hat{\nabla} F_N(\bm{\mu}_t) \|^2]\\
&\geq \mathbb{E}[F(\bm{\mu}_{t})]+\alpha_t\mathbb{E}[ \|\nabla F(\bm{\mu}_t)) \|^2] - \alpha_t^2 L G,\;\text{by Lemma \ref{variance-bound},}
\end{split}
\end{equation}
where for the first line, note that 
$$\mathbb{E}[(\nabla F(\bm{\mu}_t))' (\hat{\nabla} F(\bm{\mu}_t))]=\mathbb{E}\left[ \mathbb{E}[(\nabla F(\bm{\mu}_t))' (\hat{\nabla} F(\bm{\mu}_t)) | \bm{\mu}_t ] \right]=\mathbb{E}[ \|\nabla F(\bm{\mu}_t)) \|^2].$$
Taking the sum from $t=0$ to $t=T-1$ on both sides of (\ref{iterate}) gives
$$
\mathbb{E}[F(\bm{\mu}_{T})] \geq \mathbb{E}[F(\bm{\mu}_{0})]+\sum_{t=0}^{T-1}\alpha_t\mathbb{E}[ \|\nabla F(\bm{\mu}_t)) \|^2] -L G\sum_{t=0}^{T-1}\alpha_t^2.  \\
$$
Re-organizing the terms gives
\begin{equation*}
\begin{split}
\sum_{t=0}^{T-1} \alpha_t \mathbb{E}[\|\nabla F(\bm{\mu}_t)\|^2] &\leq \mathbb{E}[F(\bm{\mu}_{T})] -  \mathbb{E}[F(\bm{\mu}_{0})]+ L G \sum_{t=0}^{T-1} \alpha_t^2\\
&\leq f(\bm{x}^*) - F(\bm{\mu}_{0}) + L G \sum_{t=0}^\infty \alpha_t^2\\
\end{split} 
\end{equation*}
This finishes the proof for Theorem \ref{convergence-rate}.
\end{proof}
We summarize the above results in the following corollary.
\begin{corollary}
\label{summary}
Suppose Assumption \ref{f-bound} and \ref{learning-rates} hold. Given any $\delta>0$ and $\sigma>0$, there exists $N>0$ such that $F_N(\bm{\mu})$ has all its local maximums in $\mathcal{S}_{\bm{x}^*,\delta}:=\{\bm{\mu}\in\mathbb{R}^d: |\mu_i-x_i^*|\leq \delta, \text{ for each } i\in\{1,2,...,d\} \}$. For any $\varepsilon>0$, under either the PGS or EPGS setting, the updating rule (\ref{sga}) of GSPTO produces $\bm{\mu}_t$ that converges to a local maximum point of $F_N({\bm{\mu}})$, which lies in the $\delta$-neighborhood $\mathcal{S}_{\bm{x}^*,\delta}$ of $\bm{x}^*$, with the iteration complexity of $O(d^2\sigma^4\varepsilon^{-2})$. Specifically, after $T=O((d\sigma^2\varepsilon^{{-1}})^{2/(1-2\gamma)})$ times of $\bm{\mu}_t$-updating by (\ref{sga}), $\mathbb{E}[\|\nabla F(\bm{\mu}_T)\|^2]<\varepsilon$, where $\gamma\in(0,1/2)$ is a parameter in the learning rate $\alpha_t:=(t+1)^{-(1/2+\gamma)}$ and can be arbitrarily close to 0.
\end{corollary}

\section{Experiments}
\label{experiments}
\subsection{Effects of Increasing Powers}
We illustrate the improvements made by increasing $N$ for PGS/EPGS through an example problem of
\begin{equation}
\label{objective2log}
\max_{\bm{x}\in \mathbb{R}^d} f(\bm{x}):= -\log(\|\bm{x}-\bm{m}_1\|^2+10^{-5}) - \log(\|\bm{x}-\bm{m}_2\|^2+10^{-2}),
\end{equation}
the global maximum point $\bm{m}_1\in \mathbb{R}^d$ has all its entries equal to $-0.5$, and the local maximum point $\bm{m}_2\in \mathbb{R}^d$ has all its entries equal to $0.5$. The graph of its 2D-version is plotted in Figure \ref{objective-2log}.
\begin{figure}[ht]
    \centering
    \includegraphics[scale=0.45]{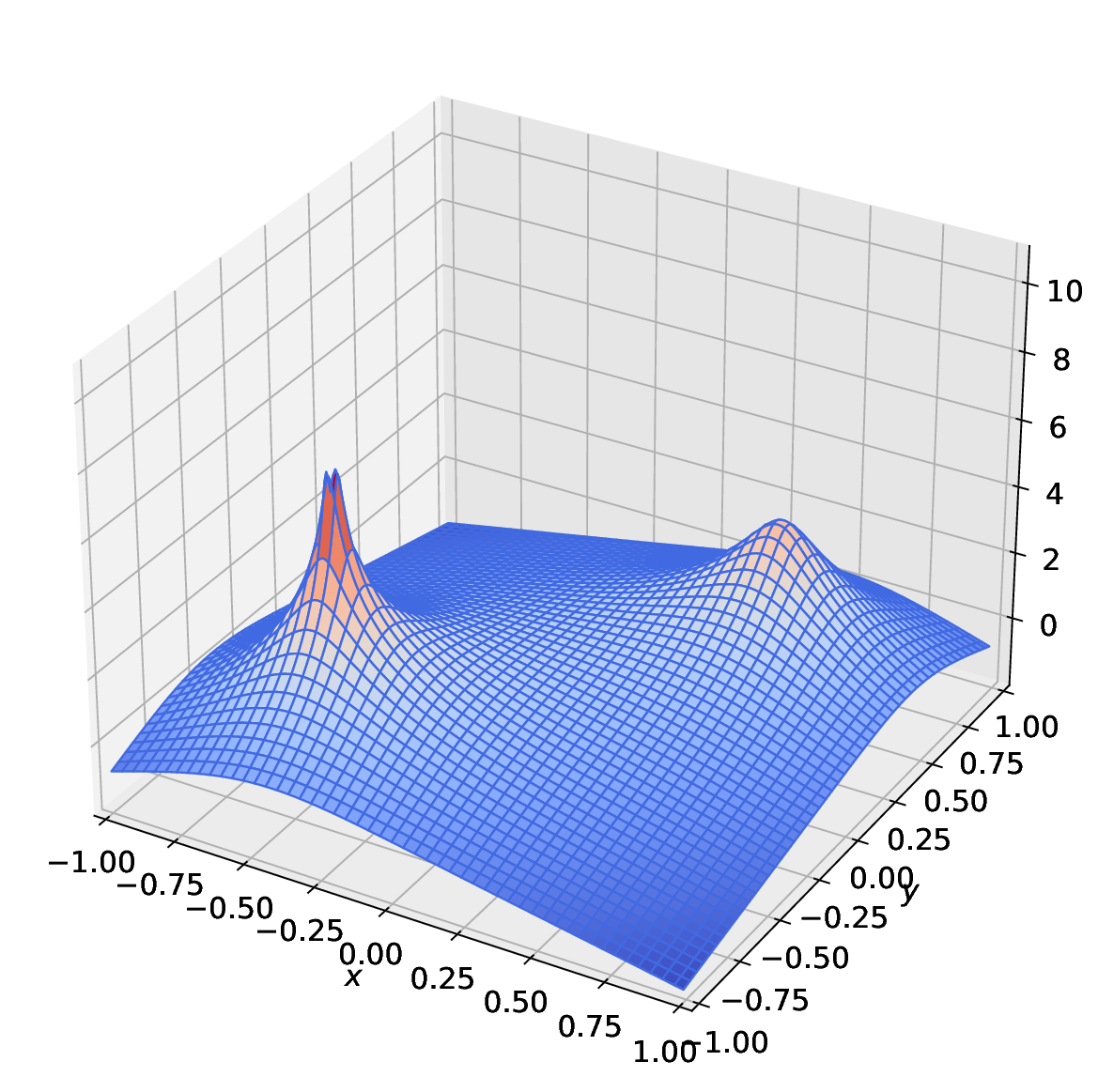}
    \caption{$f(\bm{x}) = -\log(\|\bm{x}-\bm{m}_1\|^2+10^{-5}) - \log(\|\bm{x}-\bm{m}_2\|^2+10^{-2})$}
    \label{objective-2log}
\end{figure}
  
With each value of $N$, both the PGS and EGS are performed to solve this problem. The $N$-candidate set is $\{10,20,...,65\}$ for PGS and $\{1.0, 1.5, 2.0, ..., 4.5\}$ for EGS. For each $N$ value, we do 100 trials to stabilize the result. In each trial, the initial solution candidate $\bm{\mu}_0$ is uniformly sampled from $C:=\{\bm{x}\in \mathbb{R}^d |x_i\in[-1.0,1.0], i\in\{1,2,...,d\}\}$, where $x_i$ represents the $i^{th}$ entry of $\bm{x}$. We set the initial learning rate as 0.1, the scaling parameter $\sigma$ as 0.5, and the total number of solution updates as 1000. The objective for Power-GS is modified to be $f_1(\bm{x}):=f(\bm{x})+10$ to ensure that the PGS agent will not encounter negative fitness values during the 1000 updates.

We perform the experiments in two settings, one is two-dimensional ($d=2$) and the other is five-dimensional ($d=5$). 
The results, plotted in Figure \ref{increase-N}, show that, as $N$ increases, the distance between the produced solution $\bm{\mu}^*$ and the global maximum point $\bm{x}^*$ approaches zero (see the decreasing MSE curve in the plot), which is consistent with Theorem \ref{Thm1} and the idea that $F(\bm{\mu})$'s maximum $\bm{\mu}^*$ approaches the global maximum point $\bm{x}^*$ of $f$ as we put more weight on $f(\bm{x}^*)$ 

\begin{figure}[ht]
    \centering
    \includegraphics[scale=0.45]{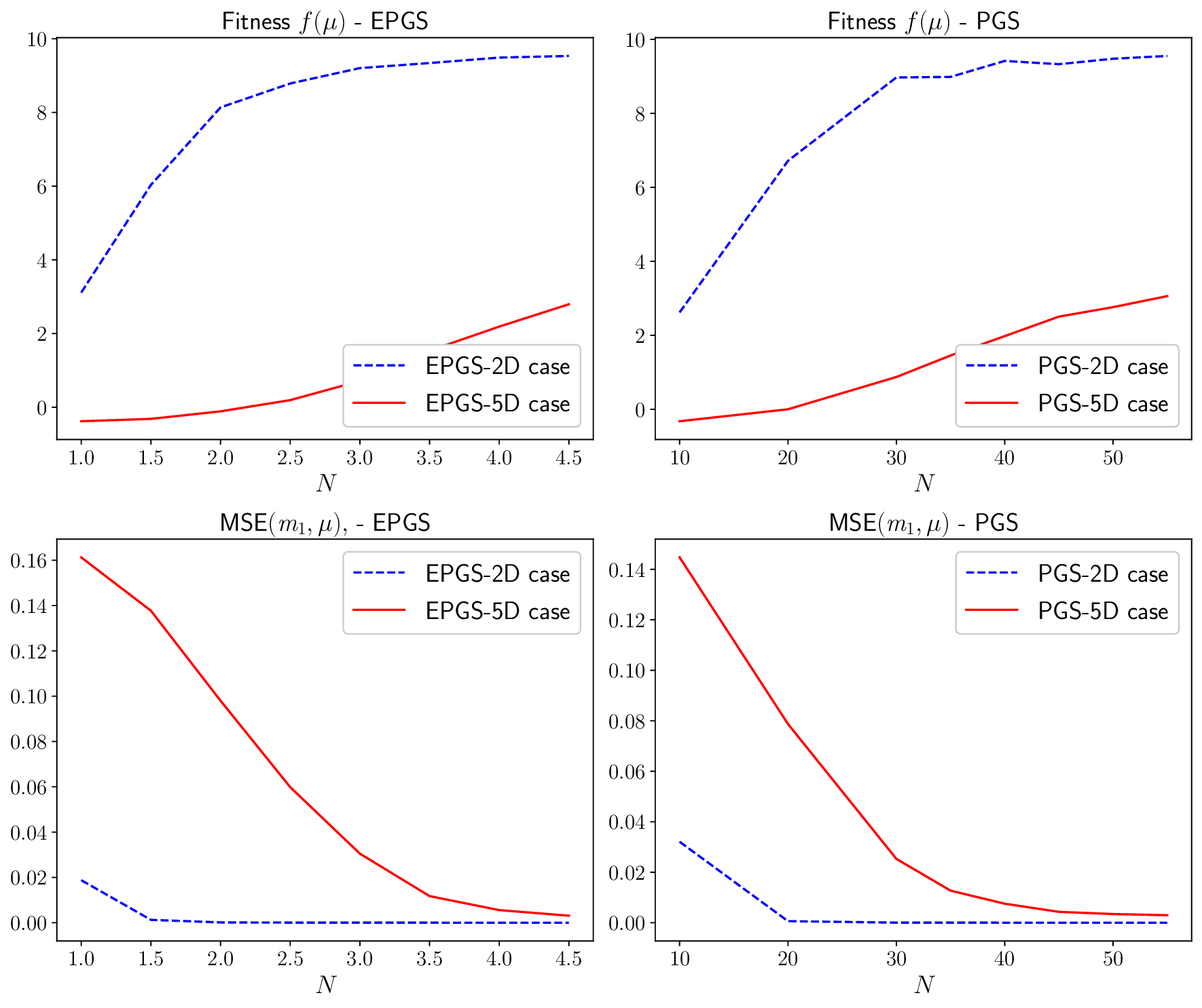}
    \caption{Effects of Increasing $N$. For each $N$, we perform the algorithm 100 times and obtain $\{\bm{\mu}_k\}_{k=1}^{100}$. The average fitness $\sum_{k=1}^{100} f(\bm{\mu}_k)/100$ and $\sum_{k=1}^{100} \text{MSE}(\bm{m}_1, \bm{\mu}_k)/100$ are plotted, where MSE$(\bm{m}_1, \bm{\mu}_k):=\sum_{i=1}^d(\mu_{ki}+0.5)^2/d$, $\sigma=1.0$, and $f$ is defined in \ref{objective2log}. Note that $\bm{x}^*=\bm{m}_1$ has all its entries equal to $-0.5$.}
    \label{increase-N}
\end{figure}

\subsection{Performance on Benchmark Objective Functions}
In this subsection, we test the performance of PGS and EGS on solving (\ref{objective}), with the objective $f$ being either of the two popular benchmark objective functions, the Ackley and the Rosenbrock (max-version). The performances of other popular global algorithms (max-version) are also reported for comparison, including 
\begin{itemize}
\item a standard homotopy method STD-Homotopy (see our appendix for details);
\item two zeroth-order single-loop Gaussian homotopy algorithms (determinsitic version\footnote{A deterministic version of SLGH is for solving the optimization problem of $\max_{\bm{x}}f(\bm{x})$, with $f$ being a deterministic function.}), ZO-SLGHd and ZO-SLGHr (\cite[Algorithm 3]{Iwakiri2022}); 
\item the gradient-ascent version of the zeroth-order algorithm of ZO-SGD (\cite[Equation (1)]{Chen2019zo} and \cite[Section 2.1]{ghadimi2013}) and ZO-AdaMM(\cite[Algorithm 1]{Chen2019zo})\footnote{ZO-SGD and ZO-AdaMM were originally designed to solves (\ref{homotopy-obj}). We take their solutions to (\ref{homotopy-obj}) as solutions to (\ref{objective}), and treat the scaling parameter in (\ref{homotopy-obj}) as a hyper-parameter.}, which were also used for comparisons in \cite{Iwakiri2022}; 
\item as well as the evolutionary algorithm of (5) particle swarm optimization (PSO, e.g., \cite[Section 3.1.5]{locatelli2013} and \cite{pyswarmsJOSS2018}). 
\end{itemize}
The details on STD-Homotopy and ZO-SGD, as well as the selected hyper-parameter values of all algorithms, can be found in our appendix. More details can be found in our codes at \url{http://github.com/chen-research/GS-PowerTransform}.

\subsubsection{Ackley}
The Ackley objective function features with a numerous number of local optimums and a single global optimum. We solve the max-version of the corresponding problem, which is 
$$\max_{(x,y)\in \mathbb{R}^2} f(x,y) := 20 e^{-0.2\sqrt{0.5(x^2+y^2)}}+e^{0.5(\cos(2\pi x)+\cos(2\pi y))}. $$
The graph of this function is plotted in Figure \ref{low-dim-obj}(a). From both the functional form and the graph, it is not difficult to see that $f(x,y)$ attains its maximum at $(0,0)$.

\begin{figure}[H]
\begin{tabular}{cc}
	\centering
        \subfloat[Ackley Function (Max-Version)]{\includegraphics[scale=0.38]{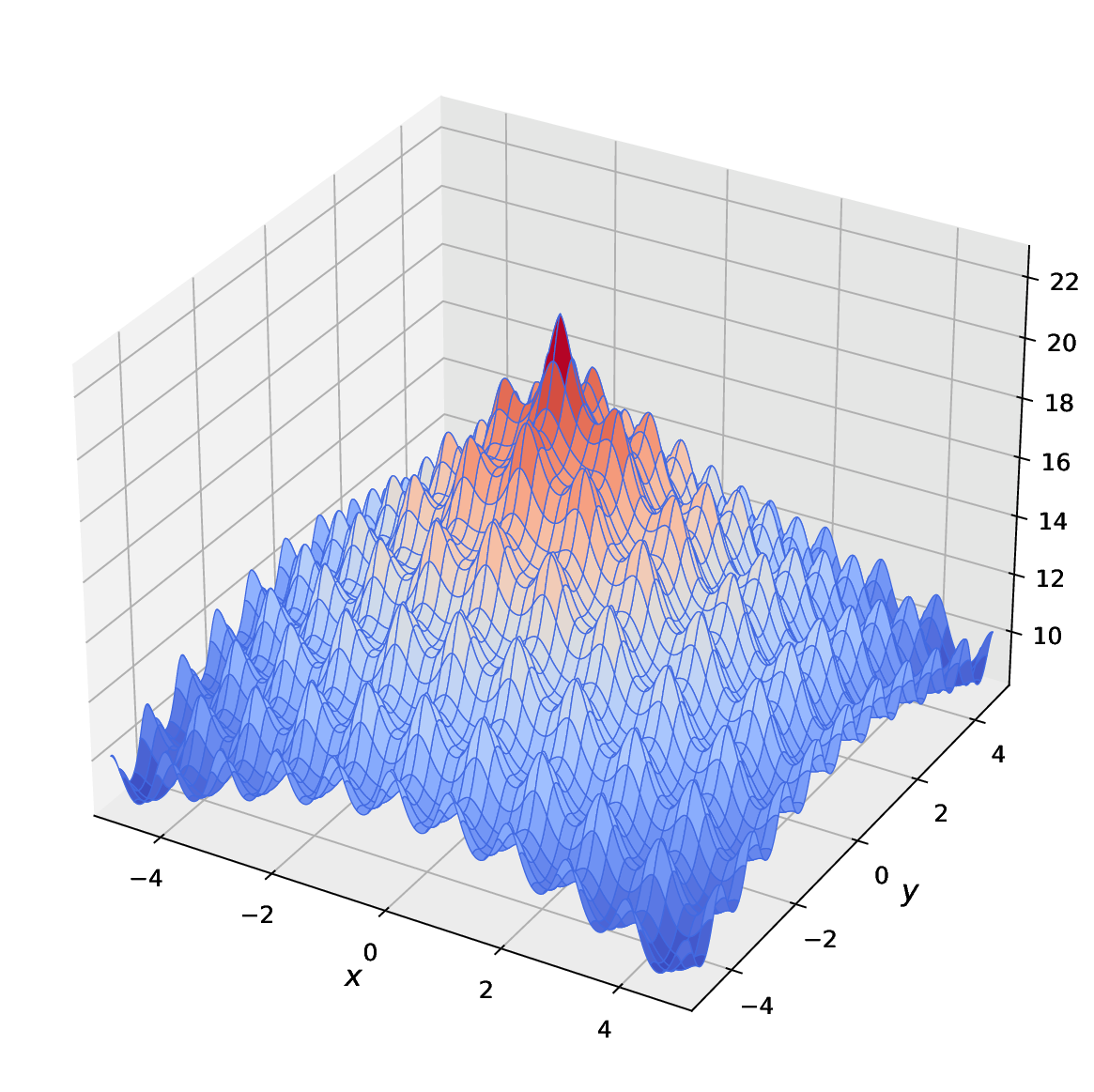} }
        &
        \subfloat[Rosenbrock Function (Max-Version).]{\includegraphics[scale=0.38]{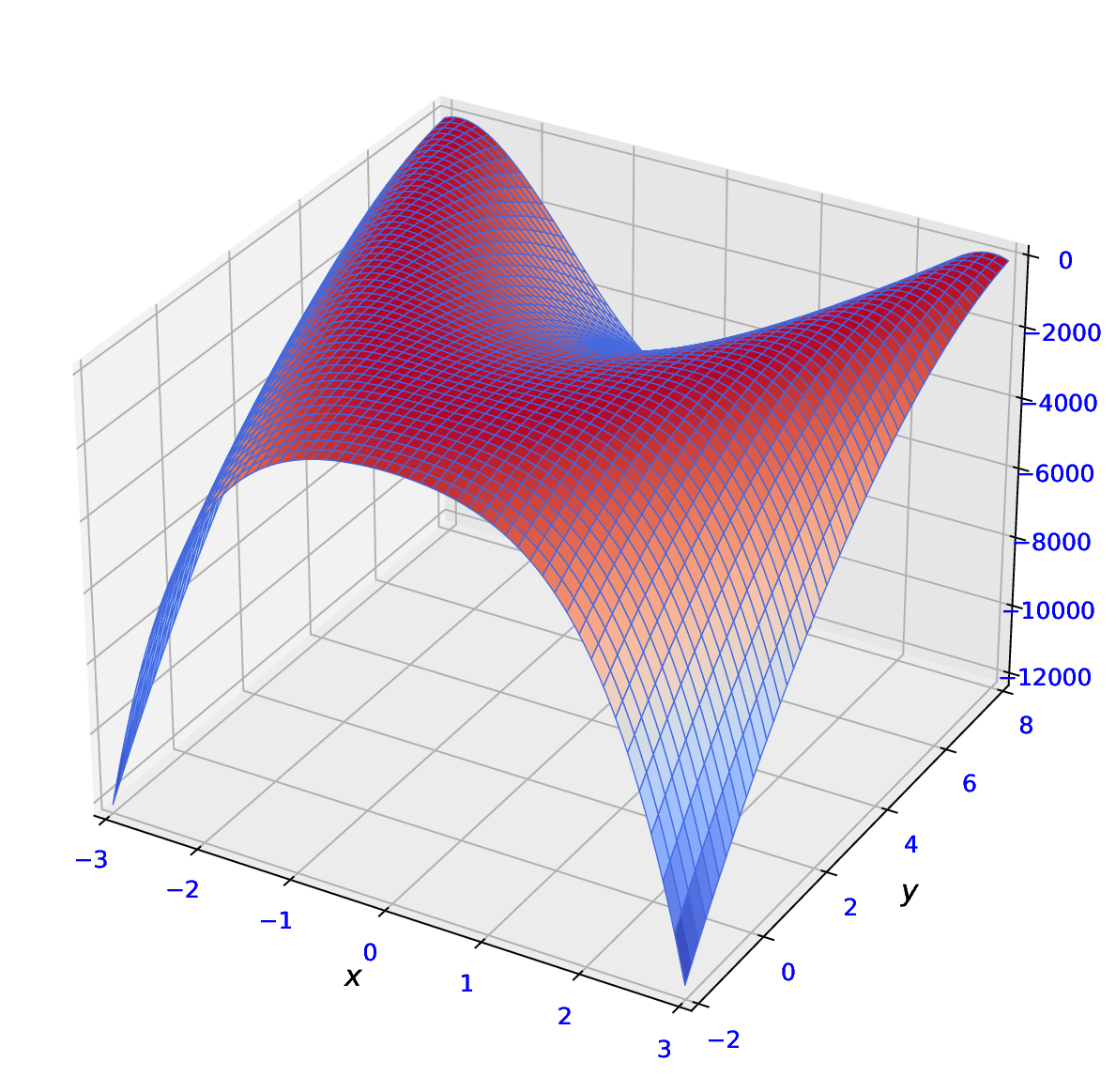} }
\end{tabular}
       	\caption{Graph of objective functions.}
\label{low-dim-obj}
\end{figure}

The solutions and their fitness values found by each of the compared algorithms are reported in Table \ref{Ackley}. From which we see that all of these algorithms are able to avoid the local maximum points and achieve the global maximum point.

\begin{table}[!htb]
\caption{Performances on Maximizing Ackley. The limit of the number of iterations is set as 200, and 100 samples are used for each solution update. The hyper-parameters are selected by trials. For each set of hyper-parameter candidates, we perform 100 experiments and \textit{take the average} to stabilize the results (all the reported numbers are averages). The initial solution value for each experiment is drawn from a multivariate Gaussian with mean $\bm{\mu}_0$ and covariance $0.01 I_2$, where $\bm{\mu}_0:=[5.0,5.0]$, except for PSO2 (for which cov$=2 I_2$). In the table, ``Iterations Taken" refers to the number of iterations taken to reach the best found solution.  All numbers are rounded to keep at most 3 decimal places. The global maximum point of the Ackley function is (0,0).}
\label{Ackley}
\centering%
\begin{tabular}{ p{4cm}  p{2.5cm}  p{5cm} p{1cm} }
\midrule
 & Iterations Taken & Best Solution Found $(\bm{\mu}^*)$  & $f(\bm{\mu}^*)$ \\
\toprule
EPGS ($N=1$) & $146$.   &  $(-0.001, -0.001)$          & 22.682\\

PGS ($N=20$)   & 141       &   $(0.001,0.002)$      &  22.678\\

STD-Homotopy   & 194       &   $(0.945,0.901)$      & 17.627\\

ZO-SLGHd   & 123       &   $(-0.003,-0.001)$.      &  22.61\\

ZO-SLGHr   & 131       &   $(0.001,-0.002)$.      &  22.621\\
ZO-AdaMM   &158       &   $(0.005, 0.001)$.      &  22.613\\
ZO-SGD   & 174       &   $(0.005, 0.007)$.      &  22.596\\
PSO1   & 174       &   $(4.425, 4.465)$.      &  10.974\\
PSO2   & 195       &   $(0.0, 0.05)$.      &  22.617\\
\bottomrule
\end{tabular}
\footnotetext{Total number of generations: 1000. $\bm{m}_1=[-.5,-.5]$. $\bm{m}_2=[.5,.5]$.}
\footnotetext{\textbf{Notation.} $\mathcal{P}_0$ denotes the initial population. For any real random vector $\bm{z}$ and any two real scalar $a<b$, $\bm{z} \sim$ uniform$[a, b]$ denotes that each entry of $\bm{z}$ is sampled uniformly from the interval $[a, b]$.}
\end{table}

\subsubsection{Rosenbrock}
The Rosenbrock objective is known to be difficult to optimize, since its global optimum point $\bm{x}^*=(1.0,1.0)$ is surrounded by a flat curved plane (see Figure). Specifically, the problem to be solved is $\max_{(x,y)\in \mathbb{R}^2} f(x,y)$, where
$$f(x,y) = -100(y-x^2)^2-(1-x)^2.$$

We use PGS, EPGS, and other algorithms to solve $\max_{\bm{x}} f(\bm{x})$. Their performances are recorded in Table \ref{Rosenbrock}, which shows that EPGS, STD-Homotopy, and PSO are superior than other algorithms on this task, since they are able to locate the true solution of (1,1).


\begin{table}[!htb]
\caption{Performances on Maximizing Rosenbrock. The limit of the number of iterations is set as 1000, and 100 samples are used for each solution update. The hyper-parameters are selected by trials. For each set of hyper-parameter candidates, we perform 100 experiments and \textit{take the average} to stabilize the results (all the reported numbers are averages). The initial solution value for each experiment is drawn from a multivariate Gaussian with mean $\bm{\mu}_0$ and covariance $0.01 I_2$, where $\bm{\mu}_0:=[-3.0,2.0]$. In the table, ``Iterations Taken" refers to the number of iterations taken to reach the best found solution.  For PGS, the Rosenbrock is added by $20,000$ to ensure the search agent only encounter positive values. All numbers are rounded to keep at most 3 decimal places. The global maximum point of the Rosenbrock function is (1,1).}
\label{Rosenbrock}
\centering%
\begin{tabular}{ p{4cm}  p{2.5cm}  p{5cm} p{1.5cm} }
\midrule
 & Iterations Taken & Best Solution Found $(\bm{\mu}^*)$  & $f(\bm{\mu}^*)$ \\
\toprule
EPGS ($N=1$) & $456$.   &  $(0.994, 1.002)$          & -0.18\\

PGS ($N=1$)  & 513       &   $(0.773,1.025)$      &  -22.8408\\

STD-Homotopy   & 624       &   $(0.903,0.885)$      &  -2.401\\

ZO-SLGHd   & 471       &   $(-0.447,1.991)$.      &  -137.016\\

ZO-SLGHr   & 148       &   $(0.105,0.938)$.      &  -88.477\\
ZO-AdaMM   & 1736       &   $(0.004, 0.618)$.      &  -39.206\\
ZO-SGD   & 45       &   $(0.272, 1.173)$.      &  -121.14\\
PSO   & 374       &   $(1.0, 1.0)$.      &  0.0\\
\bottomrule
\end{tabular}
\footnotetext{Total number of generations: 1000. $\bm{m}_1=[-.5,-.5]$. $\bm{m}_2=[.5,.5]$.}
\footnotetext{\textbf{Notation.} $\mathcal{P}_0$ denotes the initial population. For any real random vector $\bm{z}$ and any two real scalar $a<b$, $\bm{z} \sim$ uniform$[a, b]$ denotes that each entry of $\bm{z}$ is sampled uniformly from the interval $[a, b]$.}
\end{table}

\subsubsection{Conclusions on Algorithm Performances for Benchmark Objectives}
For Ackley, all the algorithms are able to locate the true solution of (1,1), except PSO when the initial population is concentrated near the initial start of $\bm{\mu}_0$, which indicates that the performance of PSO depends more on the initial guess than other methods. For Rosenbrock, EPGS outperforms other algorithms except PSO (which is quite different since it does not apply smoothing).

\subsection{Performance on the Black-box Targeted Adversarial Attack}
\label{image-attack}
Let $\mathcal{C}$ be an black-box\footnote{A black-box classifier refers to a classification model whose parameters are not accessible.} image classifier. The targeted adversarial attack on $\mathcal{C}$ refers to the task of modifying the pixels of a given image $\bm{a}$ so that $\mathcal{C}(\bm{a+x})$ is equal to a pre-specified target label $\mathcal{T}$. We perform attacks in the most difficult case - the target label $\mathcal{T}$ is pre-selected to be the one with the smallest predicted probability. Another goal of this task is to minimize the perturbation size $\|\bm{x}\|$. Hence, we set the loss as
$$ L(\bm{x}):= \max( \max_{i\neq\mathcal{T}} \mathcal{C}(\bm{a+x})_i -  \mathcal{C}(\bm{a+x})_{\mathcal{T}}, \; \kappa) + \lambda \|\bm{x}\|,$$
where $\mathcal{C}(\bm{a+x})_i$ denotes the predicted logit (i.e., log probability) for the $i^{th}$ class, $\kappa$ is a hyper-parameter that controls the certainty level of the attack, $\lambda$ is a regularization coefficient (we set $\lambda=1$ in our experiments), $\mathcal{T}:=\arg\min_{i} C(\bm{a})_i$, and $\|\bm{x}\|$ denotes the $L_1$ norm of $\bm{x}$ (i.e., the square root of the sum of squares of entries in $\bm{x}$). This loss function resembles the popular one designed in \cite{CarliniWagner2017}.

With EPGS and other compared algorithms, we perform adversarial attacks on 100 randomly selected images from each of two image datasets, the set of MNIST figures (\cite{mnist1998}) and the CIFAR-10 set (\cite{krizhevsky2009}). Specifically, the goal is to solve:
$$ \max_{\bm{x}\in\mathbb{R}^d}f(\bm{x}):= -L(\bm{x}),$$
where $d$ is the total number of pixels in each image. For Figure-MNIST images, $d=28\times28$, and for CIFAR-10 images, $d=32\times32\times 3$. We call $f$ as the fitness function.

The classifier is a robust convolutional neural network (CNN) trained using the technique of defensive distillation\footnote{We borrow the code by \cite{CarliniWagner2017} for training the classifier from \url{https://github.com/carlini/nn_robust_attacks}, which applies TensorFlow (\cite{tensorflow2015}) for model training. We changed the layer structure in the neural network to either increase accuracy or to decrease computation complexity.}. The distillation temperature is set at 100, which leads to a high level of robustness (\cite{CarliniWagner2017}). In the Figure-MNIST attacks, our trained classifier $\mathcal{C}$ has a classification accuracy of $97.4\%$ on the testing images. In the CIFAR-10 attacks, the trained $\mathcal{C}$ has a test accuracy of $87.0\%$.

The hyper-parameters of all the tested algorithms are selected by trials and can be found in Appendix \ref{Hyper-parameters}. For ZO-SLGHd, ZO-SLGHr, and ZO-AdaMM, the hyper-parameters selected in \cite[Appendix E.1]{Iwakiri2022} are included in our candidate set, since they performed similar tasks. 

We choose EPGS over PGS for this task since the fitness function $f(\bm{x})$ can be negative, which makes EPGS more convenient to apply. (But note that we can modify $f(\bm{x})$ by adding a large positive constant to facilitate PGS).

\subsubsection{MNIST}
\label{mnist_description}
For each image $\bm{a}_m$ that is randomly drawn from the testing dataset, where $m\in\{1,2,...,100\}$, and each algorithm, we perform an attack (i.e., experiment) of $T_{total}$ iterations. Let $\{\bm{\mu}_{m,t}\}_{t=0}^{T_{total}-1}$ denote all the perturbations (solutions) produced in these $T_{total}$ iterations. We say that a perturbation $\bm{\mu}$ is successful if the predicted log probability of the target label is at least $\kappa=0.01$ greater than that of other classes (i.e., $\mathcal{C}(\bm{a+\mu})_{\mathcal{T}}- \max_{i\neq\mathcal{T}} \mathcal{C}(\bm{a+\mu})_i >\kappa$). We say that an attack is successful if the produced $\{\bm{\mu}_{m,t}\}_{t=0}^{T_{total}-1}$ contains at least one successful perturbation. If the attack is successful, let $\bm{\mu}^*_m$ denote the successful perturbation with the largest $R^2$-value among $\{\bm{\mu}_{m,t}\}_{t=0}^{T_{total}-1}$, and let $T_m$ denote the number of iterations taken by the algorithm to produce $\bm{\mu}_m^*$. Here, the $R^2$-value of $\bm{\mu}$ refers to the $R^2$ statistic between $\bm{a}$ and the perturbed image $\bm{a}+\bm{\mu}$, which is computed as $\frac{\sum_{i=1}^d(a_i-\bar{a})^2}{\sum_{i=1}^d\mu_i^2}$. In this formula, $a_i$ and $\mu_i$ ranges over all the pixels (entries) of $\bm{a}$ and $\bm{\mu}$. 

With the above notations, we construct three measures on the performances of an algorithm. One is the success rate, which refers to the ratio of successful image attacks out of the total number of attacks (100). The second measure is the average $R^2$, which equals $\bar{R}^2:=\sum_{m\in\mathbb{S}} R^2(\bm{a}_m,\bm{a}_m+\bm{\mu}_m^*)/|\mathbb{S}|$, where $\mathbb{S}$ denotes the set of indices of the successful attacks. The last measure is the average $\bar{T}$ of $\{T_m\}_{m\in\mathbb{S}}$.

The results are reported in Table \ref{mnist}, from which we see that EPGS has a significantly higher $\bar{R}^2$-score (i.e., $85\%$) than other algorithms, indicating that the perturbed image is closest to the original one. Also, the average number of iterations taken by EPGS to reach the optimal perturbation is 438, which is among the two fastest algorithms. 

\begin{table}[!htb]
\caption{Targeted Adversarial Attack on 100 MNIST hand-written figures (per-image). For each image attack, the initial perturbation is set as $\bm{\mu}_0:=\bm{0}$. For each image, 1,500 iterations of are performed by each algorithm. The success rate is the portion of successful attacks out of the 100 attacks. $\bar{R}^2$ is the average of $\{R^2(\bm{a}_m,\bm{a}_m+\bm{\mu}_m^*)\}_{m\in\mathbb{S}}$, where $\bm{a}_m$ denotes the $m^{th}$ original image and $\bm{\mu}_m^*$ denotes the optimal perturbation found in the $m^{th}$ image attack, and $\mathbb{S}$ denotes the set of indices of the successful attacks. $\bar{T}$ denotes the average number of iterations taken to the optimal perturbation. The numbers in the parentheses are sample standard deviations.}
\label{mnist}
\centering%
\begin{tabular}{ >{\centering\arraybackslash} p{3cm} | >{\centering\arraybackslash}p{3cm} >{\centering\arraybackslash}p{3cm} >{\centering\arraybackslash}p{3cm} }  
\midrule
 & Success Rate & $\bar{R}^2$ & $\bar{T}$ \\
\toprule
EPGS ($N=.02$)                & $100\%$.      &  $.85\, (.05)$                 & 438\,(120)  \\
ZO-SLGHd                      &$100\%$        &  $.54\,(.14)$                   &1491\,(43)       \\
ZO-SLGHr                       &$100\%$       &  $.36\,(.26)$                    & 522\,(688)       \\
ZO-SGD                          &$100\%$       &$-3.30\,(1.67)$                & 2997\,(15)  \\
ZO-AdaMM                     &$100\%$       &$0.24\,(.26)$                &50\,(25)    \\
STD-Homotopy              &$86\%$              &$-.05$\,(.36)                &559\,(303)                                              \\
\bottomrule
\end{tabular}
\end{table}

\subsubsection{CIFAR-10}
We perform per-image targeted adversarial attacks on 100 randomly drawn images from the testing set. The experiments are performed in the same way as that for Figure-MNIST, and their results are reported in Table \ref{cifar}. These results are in general better than those in the Figure-MNIST attacks, which we believe is because of the lower accuracy of the CIFAR-10 CNN.

The results show that EPGS produces a success rate of $100\%$ and scores the second with respect to $\bar{R}^2$. Specifically, the $\bar{R}^2$-value of $98\%$ indicates that the perturbed image produced by EPGS is very close to the original image ($\bar{R}^2=100\%$ implies that the two are identical), which is consistent with our goal that the perturbation size $\|\bm{\mu}_*\|$ should be small.

In this experiment, among the algorithms that attain a $100\%$ success rate, with an $\bar{R}^2$ score of $99\%$, ZO-SLGHd outperforms other algorithms. EPGS's performance of $\bar{R}^2=98\%$ is very close to that of ZO-SLGHd.

\subsection{Summary on Experiment Results}
\label{ExperimentSummary}
The experiments show that PGS and EPGS ranked among the tops in all the tasks they performed. In the tasks of Ackley and Rosenbrock (Table \ref{Rosenbrock} and \ref{mnist}), EPGS has the highest fitness value $f(\bm{\mu}^*)$ than other algorithms that also apply smoothing techniques. For Figure-MNIST attacks (Table \ref{mnist}), EPGS has a significantly higher $\bar{R}^2$ score. For CIFAR-10 attacks (Table \ref{cifar}), EPGS's $\bar{R}^2$ score is very close to that of the best (ZO-SLGHd). In sum, EPGS outperforms other algorithms that also apply smoothing techniques in most of our experiments. 
 
\begin{table}[!htb]
\caption{Targeted Adversarial Attack on 100 CIFAR-10 images (per-image). For each image, $1,500$ iterations of are performed by each algorithm. The initial perturbation is set as $\bm{\mu}_0:=\bm{0}$. In the table, ``Succ. Rate", $\bar{R}^2$ and $\bar{T}$ are defined in the same way as in Table \ref{mnist}. All numbers are rounded to keep 3 decimal places.}
\label{cifar}
\centering%
\begin{tabular}{ >{\centering\arraybackslash} p{3cm} | >{\centering\arraybackslash}p{3cm} >{\centering\arraybackslash}p{3cm} >{\centering\arraybackslash}p{3cm} }  
\midrule
 & Success Rate & $\bar{R}^2$ & $\bar{T}$ \\
\toprule
EPGS ($N=0.03$)                & $100\%$.      &  $.98\, (.02)$                 & 686\,(277)  \\
ZO-SLGHd                      &$100\%$        &  $.99\,(.01)$                   &1288\,(429)       \\
ZO-SLGHr                       &$100\%$       &  $.98\,(.02)$                    & 402\, (307)       \\
ZO-SGD                          &$69\%$       &  $.99\,(.00)$                          & 680\,$(370)$          \\
ZO-AdaMM                     &$100\%$       & $.66 \,(.27)$                              & 64$(149)$         \\
STD-Homotopy              &$69\%$              &$.86\,(.13)$                           & 479\,(349)                     \\
\bottomrule
\end{tabular}
\end{table}

\section{Conclusion}
In this paper, we propose a novel smoothing method, GSPTO, for solving the global optimization problem of (\ref{objective}), which is featured with putting more weight on the objective's global optimum through power transformations. Both our theoretical analysis and numerical experiments show that GSPTO outperforms other optimization algorithms that also apply smoothing techniques. This method provides a foundation for future studies to explore more efficient ways to increase the gap between $f(\bm{x}^*)$ and other values before smoothing.

\bibliography{Literature}

\begin{thebibliography}{10}

\bibitem{tensorflow2015}
M.~Abadi, A.~Agarwal, P.~Barham, E.~Brevdo, Z.~Chen, C.~Citro, G.~S. Corrado,
  A.~Davis, J.~Dean, M.~Devin, S.~Ghemawat, I.~Goodfellow, A.~H., G.~Irving,
  M.~Isard, Y.~Jia, R.~Jozefowicz, L.~Kaiser, M.~Kudlur, J.~Levenberg, D.~Mane,
  R.~Monga, S.~Moore, D.~Murray, C.~Olah, M.~Schuster, J.~Shlens, B.~Steiner,
  I.~Sutskever, K.~Talwar, P.~Tucker, V.~Vanhoucke, V.~Vasudevan, F.~Viegas,
  O.~Vinyals, P.~Warden, M.~Wattenberg, M.~Wicke, Y.~Yu, and X.~Zheng.
\newblock Tensorflow: Large-scale machine learning on heterogeneous distributed
  systems, 2015.

\bibitem{BlakeZisserman1987}
A.~Blake and A.~Zisserman.
\newblock {\em Visual Reconstruction}.
\newblock MIT press, 1987.

\bibitem{CarliniWagner2017}
N.~Carlini and D.~Wagner.
\newblock Towards evaluating the robustness of neural networks.
\newblock In {\em 2017 IEEE Symposium on Security and Privacy (SP)}, pages
  39--57, 2017.

\bibitem{Chen2024}
J.~Chen, Z.~Guo, H.~Li, and C.~P. Chen.
\newblock Regularizing scale-adaptive central moment sharpness for neural
  networks.
\newblock {\em IEEE Transactions on Neural Networks and Learning Systems},
  35(5), 2024.

\bibitem{Chen2019zo}
X.~Chen, S.~Liu, K.~Xu, X.~Li, X.~Lin, M.~Hong, and D.~Cox.
\newblock Zo-adamm: Zeroth-order adaptive momentum method for black-box
  optimization.
\newblock {\em Advances in neural information processing systems}, 32, 2019.

\bibitem{Dvijotham2014}
K.~Dvijotham, M.~Fazel, and E.~Todorov.
\newblock Universal convexification via risk-aversion.
\newblock In {\em Proceedings of the Thirtieth Conference on Uncertainty in
  Artificial Intelligence}, 2014.

\bibitem{GaoSener2022}
K.~Gao and O.~Sener.
\newblock Generalizing {G}aussian smoothing for random search.
\newblock In {\em Proceedings of the 39th International Conference on Machine
  Learning}, volume 162, pages 7077--7101. PMLR, 2022.

\bibitem{ghadimi2013}
S.~Ghadimi and G.~Lan.
\newblock Stochastic first-and zeroth-order methods for nonconvex stochastic
  programming.
\newblock {\em SIAM journal on optimization}, 23(4):2341--2368, 2013.

\bibitem{Hazan2016}
E.~Hazan, K.~Y. Levy, and S.~Shalev-Shwartz.
\newblock On graduated optimization for stochastic non-convex problems.
\newblock In {\em Proceedings of The 33rd International Conference on Machine
  Learning}, volume~48, pages 1833--1841, 2016.

\bibitem{Iwakiri2022}
H.~Iwakiri, Y.~Wang, S.~Ito, and A.~Takeda.
\newblock Single loop gaussian homotopy method for non-convex optimization.
\newblock In {\em Advances in Neural Information Processing Systems},
  volume~35, pages 7065--7076. Curran Associates, Inc., 2022.

\bibitem{krizhevsky2009}
A.~Krizhevsky and G.~Hinton.
\newblock Learning multiple layers of features from tiny images.
\newblock Technical Report TR-2009, University of Toronto, 2009.

\bibitem{mnist1998}
Y.~LeCun, L.~Bottou, Y.~Bengio, and P.~Haffner.
\newblock Gradient-based learning applied to document recognition.
\newblock {\em Proceedings of the IEEE}, 86(11):2278--2324, 1998.

\bibitem{Lin2023}
X.~Lin, Z.~Yang, X.~Zhang, and Q.~Zhang.
\newblock Continuation path learning for homotopy optimization.
\newblock In {\em Proceedings of the 40th International Conference on Machine
  Learning}, volume 202, pages 21288--21311. PMLR, 2023.

\bibitem{locatelli2013}
M.~Locatelli and F.~Schoen.
\newblock {\em Global Optimization: Theory, Algorithms, and Applications}.
\newblock SIAM, Philadelphia, PA, 2013.

\bibitem{pyswarmsJOSS2018}
L.~J.~V. Miranda.
\newblock {P}y{S}warms, a research-toolkit for particle swarm optimization in
  python.
\newblock {\em Journal of Open Source Software}, 3, 2018.

\bibitem{MobahiFisher2015}
H.~Mobahi and J.~F. III.
\newblock A theoretical analysis of optimization by gaussian continuation.
\newblock In {\em Proceedings of the AAAI Conference on Artificial
  Intelligence}, volume~29, pages 1205--1211, 2015.

\bibitem{Mobahi2012}
H.~Mobahi and Y.~Ma.
\newblock Gaussian smoothing and asymptotic convexity.
\newblock Technical report, University of Illinois at Urbana-Champaign, 2012.

\bibitem{Nesterov2017}
Y.~Nesterov and V.~Spokoiny.
\newblock Random gradient-free minimization of convex functions.
\newblock {\em Foundations of Computational Mathematics}, 17(2):527--566, 2017.

\bibitem{Roulet2020}
V.~Roulet, M.~Fazel, S.~Srinivasa, and Z.~Harchaoui.
\newblock On the convergence of the iterative linear exponential quadratic
  gaussian algorithm to stationary points.
\newblock In {\em 2020 American Control Conference (ACC)}, pages 132--137.
  IEEE, 2020.

\bibitem{StarnesWebster2024}
A.~Starnes and C.~Webster.
\newblock Improved performance of stochastic gradients with gaussian smoothing,
  2024.

\bibitem{Xiao2012}
L.~Xiao and T.~Zhang.
\newblock A proximal-gradient homotopy method for the l1-regularized
  least-squares problem.
\newblock In {\em Proceedings of the 29th International Conference on Machine
  Learning (ICML-12)}, pages 839--846, 2012.

\bibitem{Xu-NIPS2016}
Y.~Xu, Y.~Yan, Q.~Lin, and T.~Yang.
\newblock Homotopy smoothing for non-smooth problems with lower complexity than
  o(1/\textbackslash epsilon).
\newblock In {\em Advances in Neural Information Processing Systems},
  volume~29. Curran Associates, Inc., 2016.

\end{thebibliography}
\bibliographystyle{abbrv}

\section{Appendix}

\subsection{Proof to Theorem \ref{Thm1} for EPGS}
\begin{proof}
Recall that for EPGS, $f_N(\bm{x}_k):=
\left\{
\begin{array}{ll}
e^{Nf(\bm{x}_k)}, &\bm{x}\in\mathcal{S};\\
0, &\text{otherwise}.\\
\end{array}
\right.$
For any given $\delta>0$, define $V_\delta:=\sup_{\bm{x}: \lVert \bm{x} - \bm{x}^*\rVert \geq \delta}f(\bm{x})$ and $D_\delta:=(V_\delta+f(\bm{x}^*))/2$. Using this symbol, we re-write $F_N(\bm{\mu},\sigma)$ as
\begin{equation}
\label{F-decompose}
F_N(\bm{\mu},\sigma) = e^{D_\delta N} G_N(\bm{\mu},\sigma) = e^{D_\delta N}(H_N(\bm{\mu},\sigma) + R_N(\bm{\mu},\sigma)),
\end{equation}
where
\begin{equation*}
\begin{split}
\label{theorem-notations}
G_N(\bm{\mu},\sigma):= & (\sqrt{2\pi}\sigma)^{-d} \int_{\bm{x}\in \mathbb{R}^d} e^{-ND_\delta} f_N(\bm{x}) e^{-\frac{\lVert \bm{x} - \bm{\mu} \rVert^2}{2\sigma^2}} d\bm{x}, \\
H_N(\bm{\mu},\sigma):= & (\sqrt{2\pi}\sigma)^{-d}  \int_{\bm{x}\in B(\bm{x}^*;\delta)} e^{-ND_\delta} f_N(\bm{x})  e^{-\frac{\lVert \bm{x} - \bm{\mu} \rVert^2}{2\sigma^2}} d\bm{x},\\
R_N(\bm{\mu},\sigma):= &  (\sqrt{2\pi}\sigma)^{-d} \int_{\bm{x}\notin B(\bm{x}^*;\delta)} e^{-ND_\delta} f_N(\bm{x})  e^{-\frac{\lVert \bm{x} - \bm{\mu} \rVert^2}{2\sigma^2}} d\bm{x},\\
\end{split}
\end{equation*}
where $B(\bm{x}^*;\delta):=\{\bm{x}\in\mathbb{R}^d: \|\bm{x}-\bm{x}^*\|<\delta \}$.

We derive an upper bound for $\left|\frac{\partial R_N(\bm{\mu},\sigma)}{\partial\mu_i}\right|$. For any $\bm{\mu}\in \mathbb{R}^d$,
\begin{equation}
\begin{split}
\label{R-bound}
\left| \frac{\partial R_N(\bm{\mu},\sigma)}{\partial\mu_i}\right| &\leq \frac{1}{(\sqrt{2\pi})^d\sigma^{d+2}}\int_{\bm{x}\notin B(\bm{x}^*;\delta)} |x_i - \mu_i| e^{-\frac{\lVert \bm{x} - \bm{\mu} \rVert^2}{2\sigma^2}} e^{-ND_\delta)}f_N(\bm{x}) d\bm{x} \\
&\leq \frac{1}{(\sqrt{2\pi})^d\sigma^{d+2}}\int_{\bm{x}\notin B(\bm{x}^*;\delta)} |x_i - \mu_i| e^{-\frac{\lVert \bm{x} - \bm{\mu} \rVert^2}{2\sigma^2}} e^{N(f(\bm{x})-D_\delta)} d\bm{x} \\
&\leq \frac{1}{(\sqrt{2\pi})^d\sigma^{d+2}}\int_{\bm{x}\notin B(\bm{x}^*;\delta)}  |x_i - \mu_i| e^{-\frac{\lVert \bm{x} - \bm{\mu} \rVert^2}{2\sigma^2}}  e^{N(V_\delta-D_\delta)}  d\bm{x} \\
&\leq  \frac{e^{N(V_\delta-D_\delta)}}{(\sqrt{2\pi})^d\sigma^{d+2}}\int_{\bm{x}\in \mathbb{R}^d}  |x_i - \mu_i| e^{-\frac{\lVert \bm{x} - \bm{\mu} \rVert^2}{2\sigma^2}}  d\bm{x} \\
&\leq e^{N(V_\delta-D_\delta)} \left( \Pi_{j\neq i}  \frac{1}{\sqrt{2\pi}\sigma} \int_{x_j\in\mathbb{R}} 
e^{-\frac{(x_j - \mu_j )^2}{2\sigma^2}}  dx_j \right)\\
&\qquad\qquad\quad  \cdot\frac{1}{\sqrt{2\pi}\sigma^{3}}\int_{x_i\in \mathbb{R}}  |x_i - \mu_i| e^{-\frac{(x_i - \mu_i)^2}{2\sigma^2}}  dx_i\\
&=e^{N(V_\delta-D_\delta)} \frac{1}{\sqrt{2\pi}\sigma^{3}}\int_{y\in \mathbb{R}} \sqrt{2}\sigma|y| e^{-y^2}  d(\sqrt{2}\sigma y), \quad y:=\frac{x_i-\mu_i}{\sqrt{2}\sigma},\\
&=e^{N(V_\delta-D_\delta)} \frac{\sqrt{2}}{\sqrt{\pi}\sigma}\cdot 2\int_{0}^{\infty}y e^{-y^2}  dy, \quad y:=\frac{x_i-\mu_i}{\sqrt{\pi}\sigma},\\
&=e^{N(V_\delta-D_\delta)} \frac{\sqrt{2}}{\sqrt{\pi}\sigma}\cdot \int_{0}^{\infty}  e^{-y^2}  d y^2,\\
&=e^{N(V_\delta-D_\delta)} \frac{\sqrt{2}}{\sqrt{\pi}\sigma}\cdot \int_{0}^{\infty}  e^{-z}  d z,\\
&=\frac{\sqrt{2}e^{N(V_\delta-D_\delta)}}{\sqrt{\pi}\sigma}
\end{split}
\end{equation}
where the third line is because $\lVert \bm{x}-\bm{x}^* \rVert \geq \delta \Rightarrow f(\bm{x})\leq V_\delta$, and the fifth line is by the separability of a multivariate integral.

Since $f$ is continuous, for $\epsilon_\delta:= f(\bm{x}^*)-D_\delta>0$ (because $V_\delta<f(\bm{x}^*)$), there exists $\delta'\in(0,\delta)$ such that whenever $\lVert \bm{x}-\bm{x}^* \rVert\leq \delta'$, 
\begin{equation}
\label{f-delta}
f(\bm{x})\geq f(\bm{x}^*)-\epsilon_\delta = D_\delta > V_\delta.
\end{equation}
Using this result, we derive a lower bound for $\left|\frac{\partial H_N(\bm{\mu},\sigma)}{\partial\mu_i}\right|$ when $\|\bm{\mu}\|\leq M$ and $|\mu_i-x_i^*|>\delta$. 
\begin{equation}
\begin{split}
\label{H-bound}
\left| \frac{\partial H_N(\bm{\mu},\sigma)}{\partial\mu_i} \right|
&= \frac{1}{(\sqrt{2\pi})^d\sigma^{d+2}}\int_{\bm{x}\in B(\bm{x}^*;\delta)} |x_i - \mu_i| e^{-ND_\delta} f_N(\bm{x})   e^{-\frac{\lVert \bm{x} - \bm{\mu} \rVert^2}{2\sigma^2}} d\bm{x} \\
&= \frac{1}{(\sqrt{2\pi})^d\sigma^{d+2}}\int_{\bm{x}\in B(\bm{x}^*;\delta)} |x_i - \mu_i| e^{-\frac{\lVert \bm{x} - \bm{\mu} \rVert^2}{2\sigma^2}}    e^{N(f(\bm{x})-D_\delta)}      d\bm{x},\text{ since } B(\bm{x}^*;\delta)\subset \mathcal{S},\\
&\geq  \frac{1}{(\sqrt{2\pi})^d\sigma^{d+2}}\int_{\bm{x}\in B(\bm{x}^*;\delta')} |x_i - \mu_i| e^{-\frac{\lVert \bm{x} - \bm{\mu} \rVert^2}{2\sigma^2}} e^{N(f(\bm{x})-D_\delta)}  d\bm{x} \\
&\geq^{\text{by } (\ref{f-delta})}  \frac{1}{(\sqrt{2\pi})^d\sigma^{d+2}}\int_{\bm{x}\in B(\bm{x}^*;\delta')} (\delta-\delta') e^{-\frac{\lVert \bm{x} - \bm{\mu} \rVert^2}{2\sigma^2}}  d\bm{x} \\
&\geq   \frac{1}{(\sqrt{2\pi})^d\sigma^{d+2}}\int_{\bm{x}\in B(\bm{x}^*;\delta')} (\delta-\delta')e^{-\frac{M^2}{\sigma^2}}e^{-\frac{\|\bm{x}\|^2}{\sigma^2}} d\bm{x},\quad\|  \bm{x} - \bm{\mu}\|^2\leq 2(\| \bm{x} \|^2+ \|\bm{\mu}\|^2),\\
&\geq  (\delta-\delta') e^{-\frac{M^2}{\sigma^2}}  V(\delta',d,\sigma) 
\end{split}
\end{equation}
where the first equality is implied by the fact that $x_i-\mu_i$ does not change sign as $\bm{x}$ travels in $B(\bm{x}^*;\delta)$ (this fact is because of $|\mu_i-x_i^*|>\delta$), and
 $$V(\delta',d,\sigma)=\frac{1}{(\sqrt{2\pi})^d\sigma^{d+2}} \int_{\bm{x}\in B(\bm{x}^*;\delta')} e^{-\frac{\|\bm{x}\|^2}{\sigma^2}} d\bm{x}.$$ 

The positive number $N_{\delta,\sigma,M}$ is constructed by solving the following inequality for $N$, which involves the two bounds in (\ref{R-bound}) and (\ref{H-bound}).
\begin{equation*}
\begin{split}
\frac{\sqrt{2}e^{N(V_\delta-D_\delta)}}{\sqrt{\pi}\sigma}< (\delta-\delta') e^{-\frac{M^2}{\sigma^2}} V(\delta',d,\sigma).
\end{split}
\end{equation*}
The solution of this inequality is
$$ N >\frac{\ln\left(\frac{\sqrt{\pi}}{\sqrt{2}}(\delta-\delta') e^{-\frac{M^2}{\sigma^2}}V(\delta',d,\sigma) \right)  }{V_\delta-D_\delta}, $$
where $V_\delta-D_\delta<0$ and the numerator is negative for sufficiently large $M>0$.
Therefore, whenever
$$N>N_{\delta,\sigma,M}:=\max\left\{0, \frac{\ln\left(\frac{\sqrt{\pi}}{\sqrt{2}}(\delta-\delta') e^{-\frac{M^2}{\sigma^2}}V(\delta',d,\sigma) \right)  }{V_\delta-D_\delta}\right\},$$
we have 
\begin{equation}
\label{H>R}
\left| \frac{\partial R_N(\bm{\mu},\sigma)}{\partial\mu_i}\right| < \frac{\sqrt{2}e^{N(V_\delta-D_\delta)}}{\sqrt{\pi}\sigma}< (\delta-\delta') e^{-\frac{M^2}{\sigma^2}} V(\delta',d,\sigma)< \left| \frac{\partial H_N(\bm{\mu},\sigma)}{\partial\mu_i} \right|.
\end{equation}
When $N>N_{\delta,\sigma,M}$, $\|\bm{\mu}\|\leq M$, and $\mu_i>x_i^*+\delta$,
\begin{equation}
\begin{split}
\label{g-bound1}
\frac{\partial G_N(\bm{\mu},\sigma)}{\partial\mu_i} &= \frac{\partial H_N(\bm{\mu},\sigma)}{\partial\mu_i}+ \frac{\partial R_N(\bm{\mu},\sigma)}{\partial\mu_i} \\
&=\frac{1}{(\sqrt{2\pi})^k\sigma^{k+2}}\int_{\bm{x}\in B(\bm{x}^*;\delta)} (x_i - \mu_i) e^{-\frac{\lVert \bm{x} - \bm{\mu} \rVert^2}{2\sigma^2}}    e^{N(f(\bm{x})-D_\delta)}      d\bm{x} + \frac{\partial R_N(\bm{\mu},\sigma)}{\partial\mu_i} \\
&= -\left|\frac{\partial H_N(\bm{\mu},\sigma)}{\partial\mu_i}\right| + \frac{\partial R_N(\bm{\mu},\sigma)}{\partial\mu_i} \\
&<^{\text{by }(\ref{H>R})} -\left|\frac{\partial R_N(\bm{\mu},\sigma)}{\partial\mu_i}\right| + \left|\frac{\partial R_N(\bm{\mu},\sigma)}{\partial\mu_i} \right|\\
&= 0,
\end{split}
\end{equation}
where the third line is because the integrand of the first term is always negative in the integration region. 

On the other hand, when $N>N_{\delta,\sigma,M}$, $\|\bm{\mu}\|\leq M$, and $\mu_i<x_i^*-\delta$,
\begin{equation}
\begin{split}
\label{g-bound2}
\frac{\partial G_N(\bm{\mu},\sigma)}{\partial\mu_i} =& \frac{\partial H_N(\bm{\mu},\sigma)}{\partial\mu_i}+ \frac{\partial R_N(\bm{\mu},\sigma)}{\partial\mu_i} \\
&= \left|\frac{\partial H_N(\bm{\mu},\sigma)}{\partial\mu_i}\right| + \frac{\partial R_N(\bm{\mu},\sigma)}{\partial\mu_i} \\
&>^{\text{by }(\ref{H>R})} \left|\frac{\partial R_N(\bm{\mu},\sigma)}{\partial\mu_i}\right| - \left|\frac{\partial R_N(\bm{\mu},\sigma)}{\partial\mu_i} \right|\\
&= 0.
\end{split}
\end{equation}
Then, (\ref{g-bound1}) and (\ref{g-bound2}) imply the result in the theorem since $\frac{\partial G_N(\bm{\mu},\sigma)}{\partial\mu_i}$ and $\frac{\partial F_N(\bm{\mu},\sigma)}{\partial\mu_i}$ share the same sign (see Eq. (\ref{F-decompose})). 
\end{proof}

\subsection{STD-Homotopy}
STD-Homotopy is a standard homotopy algorithm for optimization. It has a double-loop mechanism. The inner loop updates the solution $\bm{\mu}_t$ with a fixed scaling parameter $\sigma$ until no improvements of $f(\bm{\mu}_t)$ are made, and the outer loop decays $\sigma$ iteratively. In this algorithm, the term $\widehat{\nabla F}(\bm{\mu},\sigma)$ is an estimate of $\nabla_{\bm{\mu}} \mathbb{E}_{\bm{x}\sim\mathcal{N}(\bm{\mu},\sigma I_d)}[f(\bm{x})]= \mathbb{E}_{\bm{x}\sim\mathcal{N}(\bm{\mu},\sigma I_d)}[(\bm{x}-\bm{\mu})f(\bm{x})]$, which is used to update $\bm{\mu}$.

\begin{algorithm}[h]
  \caption{STD-Homotopy}\label{std-homotopy}
  \SetAlgoLined
  \textbf{Require}: {The maximum of iteration number $T_{total}>0$, the initial scaling parameter $\sigma>0$, the objective $f$, the initial value $\bm{\mu}_0$, the number $K$ of sampled points for gradient approximation, the maximum number $N_{\sigma}$ of times $\sigma$ gets updated,  the maximum $T_{\bm{\mu}}$ of the number of times $\bm{\mu}$ gets updated for each value of $\sigma$, the tolerance number $\tau$ for no improvements of $f(\bm{\mu}_t)$ for any fixed $\sigma$, the decay factor $\gamma\in(0,1)$, and the learning rate $\alpha_t>0$.}

  \textbf{Result}: {$\bm{\mu}_{t_1}$ - The approximated solution to (\ref{objective}).}
  
  $t_1=0$, $n_{\sigma}=0$.

  \While{$t_1\leq T_{total}$ and $n_{\sigma}<N_{\sigma}$} {
    $t=0$, $I=True$.    

    \While{$t<T_{\bm{\mu}}$ and $I==True$ and $t_1\leq T_{total}$} {
    Independently and uniformly sample $K$ points $\{\bm{v}_k\}_{k=1}^K$ from the uniform sphere in $\mathbb{R}^d$. Compute $\bm{x}_k:=\bm{\mu}_t + \sigma \bm{v}_k$, for each $k$.

    Compute the gradient estimate $ \widehat{\nabla F}(\bm{\mu}_t,\sigma)=\frac{1}{K}\sum_{k=1}^K (\bm{x}_k-\bm{\mu}_t) f(\bm{x}_k)$.\\
$ \bm{\mu}_{t+1} =\bm{\mu}_{t} + \alpha_t \widehat{\nabla F_N}(\bm{\mu}_t,\sigma)/\|\widehat{\nabla F_N,\sigma}(\bm{\mu}_t)\|$.

\If{$\max \{f(\bm{\mu}_{t+1}),f(\bm{\mu}_{t}),...,f(\bm{\mu}_{t-\tau+1})\}\leq f(\bm{\mu}_{t-\tau})$}{$I=False$.}

$t_1=t_1+1$, $t=t+1$.

}
$\sigma = \gamma \sigma$, $n_\sigma = n_\sigma+1.$
}
  Return($\bm{\mu}_{t_1}$).
\end{algorithm}

\subsection{ZO-SGD}
The zeroth-order stochastic gradient ascent (ZO-SGD) is a max-version of \cite[Equation (1)]{Chen2019zo}, whose gradient estimate method is from \cite{Nesterov2017}.

\begin{algorithm}[htb]
  \caption{ZO-SGD}\label{ZO-SGD}
  \SetAlgoLined
  \textbf{Require}: {The scaling parameter $\sigma>0$, the objective $f$, the initial value $\bm{\mu}_0\in\mathbb{R}^d$, the number $K$ of sampled points for gradient approximation, the total number $T$ of $\bm{\mu}$-updates, and the learning rate schedule $\{\alpha_t\}_{t=1}^T$.}

  \textbf{Result}: {$\bm{\mu}_T$ - The approximated solution to (\ref{objective}).}

  \For{t from 0 to T-1} {
     Independently and uniformly sample $K$ points $\{\bm{v}_k\}_{k=1}^K$ from the uniform sphere in $\mathbb{R}^d$.\\ 

Compute the gradient estimate 
$$ \widehat{\nabla F}(\bm{\mu}_t,\sigma)=\frac{1}{K}\sum_{k=1}^K \frac{(f(\bm{\mu}_t+\sigma\bm{v}_k)-f(\bm{\mu}_t))\bm{v}_kd}{\sigma}.$$

$ \bm{\mu}_{t+1} =\bm{\mu}_{t} + \alpha_t \widehat{\nabla F_N}(\bm{\mu}_t,\sigma)$.
}
  Return($\bm{\mu}_T$).
\end{algorithm}

\subsection{Hyper-parameters}
For experiments on the benchmark test functions (Table \ref{Ackley-Hyper-params} and \ref{Rosenbrock-Hyper-params}), the set of hyper-parameter values with the smallest mean square error (averaged over the 100 experiments) between the true and estimated solutions are selected. 

For the image attacks (Table \ref{MNIST-Hyper-params} and \ref{Cifar10-Hyper-params}), for each set of the hyper-parameter candidate values, we randomly choose 10 images to attack. The set with the highest average fitness value (average over the 10 image attacks) will be selected. 

\label{Hyper-parameters}
\begin{table}[!htb]
\caption{Hyper-parameters for Optimizing Ackley. The candidate set for learning rates is $\mathcal{L}:=\{.1,.001\}$. The candidate set for smoothing parameters is $\mathcal{S}:=\{.1,1.0,2.0\}$. $t_1$ in ZO-SLGHd and ZO-SLGHr is the initial scaling parameter. $\mu$ in ZO-AdaMM is the scaling parameter. The set of candidate values that lead to the minimum mean square error between the true solution and the found solution will be selected.}
\label{Ackley-Hyper-params}
\centering%
\begin{tabular}{ p{3cm}  p{4cm}  p{7cm} }
\midrule
 & Selected Values & Candidates $(\bm{\mu}^*)$ \\
\toprule
EPGS & $\alpha_0=.1$, $N=1$, $\sigma=1.0$.   &  $N\in\{1,2,3\}$, $\alpha_0\in\mathcal{L}$, $\sigma\in\mathcal{S}$.       \\

PGS   & $\alpha_0=.1$, $N=20$, $\sigma=1.0$.   &  $N\in\{5,10,20,30\}$, $\alpha_0\in\mathcal{L}$, $\sigma\in\mathcal{S}$.       \\

STD-Homotopy  & $\alpha=.1$, $\gamma=.8$, $\sigma=2.0$, $T_{\bm{\mu}}=500$, $\tau=100$, $N_\sigma=10$.   &  $\gamma\in\{.2,.5,.8\}$, $\alpha\in\mathcal{L}$, $\sigma\in\mathcal{S}$.     \\

ZO-SLGHd   & $\beta=.001$, $\eta=.001$, $t_1=2.0$, $\gamma=.99$       &   $\beta\in\{.1,.001 \}$, $\eta\in\{.1,.01,.001 \}$, $t_1\in\mathcal{S}$, $\gamma\in\{.99,.95\}$.    \\
ZO-SLGHr   & $\beta=.001$, $t_1=0.1$,  $\gamma=.995$      &   $\beta\in\mathcal{L}$, $\gamma\in\{.999,.995\}$, $t_1\in\mathcal{S}$. \\
ZO-AdaMM   &$\beta_{1}=.5$, $\beta_2=.5$, $\alpha=0.1$, $\mu=1.0$.       &   $\alpha\in\mathcal{L}$, $\beta_1\in\{.5,.7,.9\}$, $\beta_2\in\{.1,.3,.5\}$, $\mu\in\mathcal{S}$   \\
ZO-SGD   & $\alpha=.1$, $\sigma=1.0$.       &   $\alpha\in\mathcal{L}$, $\mu\in\mathcal{S}$.       \\
PSO1 and PSO2   & $c_1=.2$, $c_2=.8,$, $w=.8$.       &   $c_1,c_2,w\in\{.2,.5,.8\}$.     \\
\bottomrule
\end{tabular}
\footnotetext{Total number of generations: 1000. $\bm{m}_1=[-.5,-.5]$. $\bm{m}_2=[.5,.5]$.}
\footnotetext{\textbf{Notation.} $\mathcal{P}_0$ denotes the initial population. For any real random vector $\bm{z}$ and any two real scalar $a<b$, $\bm{z} \sim$ uniform$[a, b]$ denotes that each entry of $\bm{z}$ is sampled uniformly from the interval $[a, b]$.}
\end{table}

\begin{table}[!htb]
\caption{Hyper-parameters for Optimizing Rosenbrock. The candidate set for learning rates is $\mathcal{L}:=\{.2,.1,.01,.001,.0001\}$. The candidate set for smoothing parameters is $\mathcal{S}:=\{1.0,2.0\}$.  $t_1$ in ZO-SLGHd and ZO-SLGHr is the initial scaling parameter. $\mu$ in ZO-AdaMM is the scaling parameter. For EPGS and PGS, $\alpha_t = \frac{1,000\alpha_0}{1,000+t}$. The set of candidate values that lead to the minimum mean square error between the true solution and the found solution will be selected.}
\label{Rosenbrock-Hyper-params}
\centering%
\begin{tabular}{ p{3cm}  p{4cm}  p{7cm} }
\midrule
 & Selected Values & Candidates $(\bm{\mu}^*)$ \\
\toprule
EPGS & $\alpha_0=.2$, $N=1$, $\sigma=1.0$.   &  $N\in\{1,2,3\}$, $\alpha_0\in\mathcal{L}$, $\sigma\in\mathcal{S}$.       \\

PGS  & $\alpha_0=.1$, $N=1$, $\sigma=1.0$.   &  $N\in\{1,3,5\}$, $\alpha_0\in\mathcal{L}$, $\sigma\in\mathcal{S}$.       \\

STD-Homotopy  & $\alpha=.2$, $\gamma=.2$, $\sigma=2.0$,, $T_{\bm{\mu}}=500$, $\tau=100$, $N_\sigma=10$.   &  $\gamma\in\{.2,.5,.8\}$, $\alpha\in\mathcal{L}$, $\sigma\in\mathcal{S}$.     \\

ZO-SLGHd   & $\beta=.0001$, $\eta=.001$, $t_1=2.0$, $\gamma=.999$       &   $\beta\in\mathcal{L}$, $\eta\in\{.1,.01,.001 \}$, $t_1\in\mathcal{S}$, $\gamma\in\{.99,.995,.999\}$.    \\
ZO-SLGHr   & $\beta=.0001$, $t_1=2.0$,  $\gamma=.999$.      &   $\beta\in\mathcal{L}$, $\gamma\in\{.999,.995\}$, $t_1\in\mathcal{S}$. \\
ZO-AdaMM   &$\beta_{1}=.5$, $\beta_2=.5$, $\alpha=.2$, $\mu=2.0$.       &   $\alpha\in\mathcal{L}$, $\beta_1\in\{.5,.7,.9\}$, $\beta_2\in\{.1,.3,.5\}$, $\mu\in\mathcal{S}$   \\
ZO-SGD   & $\alpha=.001$, $\sigma=2.0$.       &   $\alpha\in\mathcal{L}$, $\mu\in\mathcal{S}$.       \\
PSO1 and PSO2   & $c_1=.2$, $c_2=.5,$, $w=.8$.       &   $c_1,c_2,w\in\{.2,.5,.8\}$.     \\
\bottomrule
\end{tabular}
\footnotetext{Total number of generations: 1000. $\bm{m}_1=[-.5,-.5]$. $\bm{m}_2=[.5,.5]$.}
\footnotetext{\textbf{Notation.} $\mathcal{P}_0$ denotes the initial population. For any real random vector $\bm{z}$ and any two real scalar $a<b$, $\bm{z} \sim$ uniform$[a, b]$ denotes that each entry of $\bm{z}$ is sampled uniformly from the interval $[a, b]$.}
\end{table}

\begin{table}[!htb]
\caption{Hyper-parameters for Figure-MNIST Attack. The candidate set for (initial) smoothing parameters is $\mathcal{S}:=\{1.0, .1\}$. The hyper-parameter symbols for each algorithm are the same as their source publications. For example, $t_1$ in ZO-SLGHd and ZO-SLGHr denotes the initial scaling parameter, $\mu$ in ZO-AdaMM is the scaling parameter, and $\alpha$ denotes a constant learning rate. The number $784$ equals the dimensional number $d$. It appears in the candidate set since it is taken from \cite{Iwakiri2022}, who performed similar experiments. The set of candidate values that lead to the highest fitness (averaged over the 10 image attackes) are be selected.}
\label{MNIST-Hyper-params}
\centering%
\begin{tabular}{ p{3cm}  p{4cm}  p{7cm} }
\midrule
 & Selected Values & Candidates $(\bm{\mu}^*)$ \\
\toprule
EPGS & $\alpha =.1$, $N=.02$, $\sigma=.1$.   &  $N\in\{.01,.02,.03\}$, $\alpha_t\in\{.1,.05\}$, $\sigma\in\mathcal{S}$.       \\

STD-Homotopy  & $\alpha=.5$, $\gamma=.5$, $\sigma=1.0$,, $T_{\bm{\mu}}=500$, $\tau=100$, $N_\sigma=10$.   &  $\gamma\in\{.5,.8\}$, $\alpha\in\{.5,.1\}$, $\sigma\in\mathcal{S}$.     \\

ZO-SLGHd   & $\beta=10^{-4}$, $\eta=.001$, $t_1=.1$, $\gamma=.995$      &   $\beta\in\{1/784,10^{-4}, .1\}$, $\eta\in\{.1/784,10^{-3}\}$, $t_1\in\mathcal{S}$, $\gamma\in\{.999,.995\}$.    \\
ZO-SLGHr   & $\beta=10^{-4}$, $t_1=.1$,  $\gamma=.995$.     &   $\beta\in\{10^{-4},1/784,.1 \}$, $\gamma\in\{.999,.995\}$, $t_1\in\mathcal{S}$. \\
ZO-AdaMM   &$\beta_{1}=.9$, $\beta_2=.1$, $\alpha=.1$, $\mu=.1$.        &   $\alpha\in\{100/784,.001,.1\}$, $\beta_1\in\{.5,.9\}$, $\beta_2\in\{.1,.3\}$, $\mu\in\mathcal{S}$   \\
ZO-SGD   & $\alpha=10^{-4}$, $\mu=.1$.      &   $\alpha\in\{10^{-4},.1,1/784\}$, $\mu\in\mathcal{S}$.       \\
\bottomrule
\end{tabular}
\footnotetext{Total number of generations: 1000. $\bm{m}_1=[-.5,-.5]$. $\bm{m}_2=[.5,.5]$.}
\footnotetext{\textbf{Notation.} $\mathcal{P}_0$ denotes the initial population. For any real random vector $\bm{z}$ and any two real scalar $a<b$, $\bm{z} \sim$ uniform$[a, b]$ denotes that each entry of $\bm{z}$ is sampled uniformly from the interval $[a, b]$.}
\end{table}

\begin{table}[!htb]
\caption{Hyper-parameters for CIFAR-10 Attack. The candidate set for (initial) smoothing parameters is $\mathcal{S}:=\{1.0, .1\}$. The symbols are the same as those in Table \ref{MNIST-Hyper-params}. The set of candidate values that lead to the highest fitness (averaged over the 10 image attackes) are be selected.}
\label{Cifar10-Hyper-params}
\centering%
\begin{tabular}{ p{3cm}  p{4cm}  p{7cm} }
\midrule
 & Selected Values & Candidates $(\bm{\mu}^*)$ \\
\toprule
EPGS & $\alpha=.1$, $N=.03$, $\sigma=.1$.   &  $N\in\{.02,.03,.04\}$, $\alpha_t\in\{.1,.05\}$, $\sigma\in\mathcal{S}$.       \\

STD-Homotopy  & $\alpha=.5$, $\gamma=.8$, $\sigma=.1$, $T_{\bm{\mu}}=300$, $\tau=100$, $N_\sigma=10$.   &  $\gamma\in\{.5,.8\}$, $\alpha\in\{.1,.5\}$, $\sigma\in\mathcal{S}$.     \\

ZO-SLGHd   & $\beta=.01/3072$, $\eta=10^{-5}$, $t_1=1.0$, $\gamma=.999$       &   $\beta\in\{.01/3072, .1\}$, $\eta\in\{10^{-4}/784,10^{-5} \}$, $t_1\in\mathcal{S}$, $\gamma\in\{.995,.999\}$.    \\
ZO-SLGHr   & $\beta=.01/3072$, $t_1=.1$,  $\gamma=.995$.      &   $\beta\in\{.01/3072,.1\}$, $\gamma\in\{.999,.995\}$, $t_1\in\mathcal{S}$. \\
ZO-AdaMM   &$\beta_{1}=.9$, $\beta_2=.1$, $\alpha=.1$, $\mu=.1$.       &   $\alpha\in\{.5/3072,.001,.1\}$, $\beta_1\in\{.5,.9\}$, $\beta_2\in\{.1,.3\}$, $\mu\in\mathcal{S}$   \\
ZO-SGD   & $\alpha=.01/3072$, $\sigma=.1$.       &   $\alpha\in\{.01/3072,10^{-4},.05\}$, $\sigma\in\mathcal{S}$.       \\
\bottomrule
\end{tabular}
\footnotetext{Total number of generations: 1000. $\bm{m}_1=[-.5,-.5]$. $\bm{m}_2=[.5,.5]$.}
\footnotetext{\textbf{Notation.} $\mathcal{P}_0$ denotes the initial population. For any real random vector $\bm{z}$ and any two real scalar $a<b$, $\bm{z} \sim$ uniform$[a, b]$ denotes that each entry of $\bm{z}$ is sampled uniformly from the interval $[a, b]$.}
\end{table}

\end{document}